\documentclass[11pt,a4paper]{amsart}

\usepackage[T1]{fontenc}
\usepackage{lmodern}
\usepackage{amsmath,amssymb,mathtools}
\usepackage{booktabs,array,tabularx}
\usepackage{microtype}
\usepackage{xcolor}
\usepackage[colorlinks=true,linkcolor=blue,citecolor=blue,urlcolor=blue,hypertexnames=false]{hyperref}
\hypersetup{
  pdftitle={Simple and distinct zeros in a prime-modulus Dirichlet family from near-microscopic to polylogarithmic heights},
  pdfauthor={Zhixu Hua, Xiufan Yang}
}

\allowdisplaybreaks

\newtheorem{theorem}{Theorem}[section]
\newtheorem{lemma}[theorem]{Lemma}
\newtheorem{proposition}[theorem]{Proposition}
\newtheorem{corollary}[theorem]{Corollary}
\theoremstyle{definition}

\theoremstyle{remark}
\newtheorem{remark}[theorem]{Remark}

\DeclareMathOperator{\dist}{dist}
\DeclareMathOperator{\rank}{rank}
\DeclareMathOperator{\tr}{tr}
\DeclareMathOperator{\supp}{supp}
\newcommand{\R}{\mathbb R}
\newcommand{\C}{\mathbb C}
\newcommand{\Z}{\mathbb Z}
\newcommand{\one}{\mathbf 1}
\newcommand{\norm}[1]{\lVert #1\rVert}
\newcommand{\eps}{\varepsilon}
\newcolumntype{Y}{>{\raggedright\arraybackslash}X}

\title[Simple and distinct zeros in an extended height range]{Simple and distinct zeros in a prime-modulus Dirichlet family from near-microscopic to polylogarithmic heights}
\author{Zhixu Hua}
\address{Changkong College, Nanjing University of Aeronautics and Astronautics, Nanjing, Jiangsu, China}
\email{hzx001@nuaa.edu.cn}
\author{Xiufan Yang}
\address{School of Science, Nanjing University of Posts and Telecommunications, Nanjing, Jiangsu, China}
\email{b25100020@njupt.edu.cn}
\thanks{Zhixu Hua and Xiufan Yang contributed equally to this work.}
\date{}

\begin{document}


\begin{abstract}
Fix \(\eta>0\) and \(A_0>0\). Let \(q\) tend to infinity through odd
primes, put \(Q=\log q\), and let \(T=T(q)\) satisfy
\[
 \frac{(\log Q)^{1+\eta}}{Q}\le T\le Q^{A_0}.
\]
Set \(I=(T,2T]\). Summing without weights over the \(q-2\)
nonprincipal characters modulo \(q\), let \(\mathcal N_q\) be the total
number of nontrivial zeros in \(I\), counted with multiplicity. Let
\(\mathcal N^s_{0,q}\) and \(\mathcal N^*_{0,q}\) be the corresponding
totals of simple zeros and distinct zeros on the critical line,
respectively, and let \(\mathcal N_{d,q}\) count all distinct zeros in
\(I\), without restriction to the critical line. Uniformly throughout the
displayed height range, we prove unconditionally that
\[
\begin{gathered}
 \mathcal N_q=\frac{qT\log q}{2\pi}\{1+o_{\eta,A_0}(1)\},\\
 \frac{\mathcal N^s_{0,q}}{\mathcal N_q}
 \ge C_{\mathrm{MT}}-o_{\eta,A_0}(1),\qquad
 \frac{\mathcal N^*_{0,q}}{\mathcal N_q}
 \ge C_{\mathrm{MT}}-o_{\eta,A_0}(1),\\
 \frac{\mathcal N_{d,q}}{\mathcal N_q}
 \ge C_d-o_{\eta,A_0}(1),
\end{gathered}
\]
where
\[
\begin{aligned}
 C_{\mathrm{MT}}
 &=\frac32-\frac1{\sqrt2}\cot\!\left(\frac1{\sqrt2}\right)
   =0.672500703679\ldots,\\
 C_d&=\frac{1+C_{\mathrm{MT}}}{2}=0.836250351839\ldots.
\end{aligned}
\]
Here \(o_{\eta,A_0}(1)\to0\) as \(q\to\infty\), uniformly in \(T\).
The normalized interval length is \(T\log q/(2\pi)\); thus the result
applies to intervals containing at least
\((\log\log q)^{1+\eta}/(2\pi)\) local mean spacings, while the physical
height may tend to zero, remain bounded, or grow as any prescribed fixed
power of \(\log q\).
The proof combines Selberg's family-averaged argument estimate and a
zero-density deletion with a finite Gevrey Gabor compression of Weil's
Hermitian form. A quantitative Fourier--Laplace estimate gives uniform
trace-norm control of the exterior zeros at the lower end of the stated
range, and a local--remote shell decomposition gives the same control
throughout the polylogarithmic upper range. First and second matrix moments,
together with an inertia-based rank--trace inequality, yield the three
counting bounds. No form of the generalized Riemann hypothesis is assumed.
\end{abstract}

\maketitle
\section{Introduction}

Fix \(\eta>0\) and \(A_0>0\), let \(q\) be an odd prime, and write
\[
 Q=\log q,\qquad
 \frac{(\log Q)^{1+\eta}}{Q}\le T\le Q^{A_0},\qquad
 \mathcal H=TQ,\qquad I=(T,2T].
\]
The height \(T=T(q)\) may vary with \(q\). For a nonprincipal character
\(\chi\pmod q\), write a nontrivial zero of \(L(s,\chi)\) as
\(\rho=\beta+i\gamma\). The normalized ordinate
\(x_\rho=Q\gamma/(2\pi)\) maps \(I\) to an interval of length
\(\mathcal H/(2\pi)\). We call this range mesoscopic because
\(\mathcal H\ge(\log Q)^{1+\eta}\to\infty\), while the physical height is
allowed to range from
\(Q^{-1}(\log Q)^{1+\eta}\) to an arbitrary fixed polylogarithmic
height. In particular, every fixed-power
scale \(T=Q^\xi\) with \(\xi>-1\) is covered after choosing \(A_0\)
large enough. The first arXiv version treated the shrinking range
\(T=Q^{-\alpha}\), \(0<\alpha<1\). The present revision establishes the
uniform range displayed above.

Theorem~\ref{thm:main} concerns three family statistics, all normalized by
total zero multiplicity in \(I\). It gives the Montgomery--Taylor lower
bound
\[
 C_{\mathrm{MT}}
 =\frac32-\frac1{\sqrt2}\cot\!\left(\frac1{\sqrt2}\right)
 =0.672500703679\ldots
\]
for both simple zeros and distinct zeros on the critical line, and the
bound
\[
 C_d=\frac{1+C_{\mathrm{MT}}}{2}
 =\frac54-\frac1{2\sqrt2}\cot\!\left(\frac1{\sqrt2}\right)
 =0.836250351839\ldots
\]
for all distinct zeros in \(I\), without restriction to the critical line.
The common denominator is asymptotic to \(qTQ/(2\pi)\), uniformly in the
stated height range. The theorem concerns the unweighted family sum over all
nonprincipal characters of one prime modulus; it is not a
character-by-character assertion.

Simple-zero proportions have traditionally been studied through
pair-correlation or Levinson--mollifier methods; representative Dirichlet
family results include
\cite{ChandeeLeeLiuRadz2014,ConreyIwaniecSoundararajan2013,Wu2016,Sono2025},
while fixed normalized low-lying statistics and first-zero questions belong
to different normalizations
\cite{HughesRudnick2003,ZhaoPositivity2026,Zhao2026,HiaryZhao2026}.
The finite compression of Weil's Hermitian form, the treatment of off-line
orbits by inertia, and the rank--trace certificate first appeared in an
Anthropic manuscript \cite{Claude2026Current}. Alp\"oge and Furman
subsequently gave a verified account for the Riemann zeta function and
fixed primitive Dirichlet \(L\)-functions, together with a Lean~4
formalization \cite{AlpogeFurman2026}. An earlier manuscript had sketched a
possible character-family extension at polylogarithmic heights without
carrying it out \cite{AnthropicRiemann2026,Claude2026Superseded}.

After the first version of the present paper appeared, Fredrik Pr\"uzelius
drew our attention to \cite{AlpogeFurman2026} and communicated a
contemporaneous manuscript on the polylogarithmic-height regime, archived at
\url{https://doi.org/10.5281/zenodo.21980224}. That communication prompted us
to re-examine the height dependence of our argument. The proof in this
revision is self-contained and establishes the one-prime-modulus result
uniformly in the range
\[
 \frac{(\log\log q)^{1+\eta}}{\log q}
 \le T\le(\log q)^{A_0}.
\]
The additional inputs beyond the finite-compression certificate are
Selberg's averaged argument estimate, the Hiary--Zhao zero-density deletion
of \(O(q/Q^2)\) exceptional characters, uniform finite-centre moments, and a
Gevrey Fourier--Laplace trace-norm estimate. The variational constant is
classical
\cite{Montgomery1975,ConreyGhoshGonek1998}. Further low-lying density,
conditional simple-zero, pair-correlation, and high-height results include
\cite{DrappeauPrattRadziwill2023,Sono2016,Wu2019,Dickinson2024,
GarunkstisPaliulionyte2025,ConreyKwanLinTurnageButterbaugh2026}; their
averaging variables, height regimes, or denominators differ from those here.
Selberg's averaged argument estimate and the Hiary--Zhao zero-density theorem
provide the family input used below
\cite{Selberg1946,HiaryZhao2026}.

The proof has four main steps. First, an endpoint-safe family counting
formula and a zero-density shell argument give the denominator and the
good--bad character decomposition. Second, the explicit formula and exact
character orthogonality for prime polynomials of length \(q^\lambda<q\)
evaluate the first two moments of a finite Gabor matrix. Third,
critical-line zeros and off-line functional-equation orbits are encoded by
blocks with controlled positive index, while the exterior zero set is
removed in trace norm. The localization step uses a self-contained Gevrey
Fourier--Laplace estimate: near the critical line it gives subexponential
decay in \(TQ\), while away from the line it is combined with goodness and
a local--remote decomposition of the ordinate shells.
Finally, a rank--trace inequality gives the certificates for simple and
distinct zeros, and a one-dimensional variational problem yields
\(C_{\mathrm{MT}}\). The fixed-parameter reduction appears in Section~2;
the counting and zero-density inputs are stated in Section~3; and
Section~12 records the uniform error scales and the limitation at the
fixed-normalized-window endpoint. Throughout, \(\eta,A_0\) and every auxiliary parameter are fixed
before \(q\to\infty\), and all little-oh terms are uniform over the
admissible choices of \(T\).

\section{Main theorem, notation, and reduction}

\subsection{Main theorem and counting conventions}
Fix \(\eta>0\) and \(A_0>0\). Throughout the paper, \(q\) tends to
infinity through odd primes, while \(T=T(q)\) may vary subject to
\begin{equation}
\tag{2.1}\label{eq:2.1}
Q=\log q,\qquad
\frac{(\log Q)^{1+\eta}}{Q}\le T\le Q^{A_0},\qquad
\mathcal H=TQ,\qquad I=(T,2T].
\end{equation}
For a nonprincipal character \(\chi\pmod q\), let
\begin{equation}
\tag{2.2}\label{eq:2.2}
\mathcal N_\chi(I)
 =\sum_{\substack{L(\rho,\chi)=0\\ T<\Im\rho\le2T}}m_\rho
\end{equation}
be the number of nontrivial zeros in \(I\), counted with multiplicity. Let
\(\mathcal N^s_{0,\chi}(I)\) denote the number of zeros in the same interval
which lie on \(\Re s=1/2\) and have multiplicity one. Define
\begin{equation}
\tag{2.2a}\label{eq:2.2a}
\begin{aligned}
\mathcal N^*_{0,\chi}(I)
 &=\#\{\gamma\in I:L(1/2+i\gamma,\chi)=0\},\\
\mathcal N_{d,\chi}(I)
 &=\#\{\rho:L(\rho,\chi)=0,\ \rho\ \text{nontrivial},\
                 \ T<\Im\rho\le2T\}.
\end{aligned}
\end{equation}
Thus \(\mathcal N^*_{0,\chi}\) counts each distinct zero on the critical
line once, including a multiple zero only once, while
\(\mathcal N_{d,\chi}\) counts every distinct zero in \(I\) once, without
restriction to the critical line. Define the family counts
\begin{equation}
\tag{2.3}\label{eq:2.3}
\begin{aligned}
\mathcal N_q&=\sum_{\chi\ne\chi_0}\mathcal N_\chi(I),&
\mathcal N^s_{0,q}&=\sum_{\chi\ne\chi_0}\mathcal N^s_{0,\chi}(I),\\
\mathcal N^*_{0,q}&=\sum_{\chi\ne\chi_0}\mathcal N^*_{0,\chi}(I),&
\mathcal N_{d,q}&=\sum_{\chi\ne\chi_0}\mathcal N_{d,\chi}(I).
\end{aligned}
\end{equation}
The half-open convention in \eqref{eq:2.1} is used everywhere, including
the zero matrix and its complementary tail. When the symmetric counting
function \(N(t,\chi)\) is used, the interval is extracted by a common right
shift of both endpoints as in Lemma~\ref{lem:half-open-extraction}; no
assumption that \(T\) or \(2T\) is not a zero ordinate is made.

\begin{theorem}[Simple and distinct zeros from near-microscopic to polylogarithmic heights]
\label{thm:main}
For every fixed \(\eta>0\) and \(A_0>0\), uniformly for all \(T\)
satisfying \eqref{eq:2.1},
\begin{equation}
\tag{2.4}\label{eq:2.4}
\frac{\mathcal N^s_{0,q}}{\mathcal N_q}
\ge C_{\mathrm{MT}}-o_{\eta,A_0}(1),
\qquad
C_{\mathrm{MT}}
:=\frac32-\frac1{\sqrt2}\cot\!\left(\frac1{\sqrt2}\right).
\end{equation}
The corresponding distinct-zero bounds are
\begin{equation}
\tag{2.4a}\label{eq:2.4a}
\begin{aligned}
\frac{\mathcal N^*_{0,q}}{\mathcal N_q}
&\ge C_{\mathrm{MT}}-o_{\eta,A_0}(1),\\
\frac{\mathcal N_{d,q}}{\mathcal N_q}
&\ge C_d-o_{\eta,A_0}(1),\qquad
C_d:=\frac{1+C_{\mathrm{MT}}}{2}
 =\frac54-\frac1{2\sqrt2}\cot\!\left(\frac1{\sqrt2}\right).
\end{aligned}
\end{equation}
Moreover, uniformly in the same range,
\begin{equation}
\tag{2.5}\label{eq:2.5}
\mathcal N_q=\frac{qT\log q}{2\pi}\{1+o_{\eta,A_0}(1)\}.
\end{equation}
Here every \(o_{\eta,A_0}(1)\) tends to zero as \(q\to\infty\), uniformly
in \(T\) throughout \eqref{eq:2.1}.
\end{theorem}

\begin{remark}[Height range and counting conventions]
\label{rem:ratio-interpretation}
Every fixed-power scale \(T=Q^\xi\) with \(\xi>-1\) is permitted after
choosing \(A_0>\max\{\xi,0\}\). The range also contains scales much closer
to \(Q^{-1}\), down to
\[
 T=Q^{-1}(\log Q)^{1+\eta}.
\]
In the normalized coordinate \(x=Q\gamma/(2\pi)\), the interval is
\((\mathcal H/(2\pi),\mathcal H/\pi]\) and has length at least
\((\log Q)^{1+\eta}/(2\pi)\). The theorem is not claimed at the exact
fixed-normalized-window scale \(TQ\asymp1\).

The denominator in \eqref{eq:2.4}--\eqref{eq:2.4a} counts every zero in
\(I\) with its multiplicity. A simple zero on the critical line contributes
one to each of \(\mathcal N_\chi\), \(\mathcal N^s_{0,\chi}\),
\(\mathcal N^*_{0,\chi}\), and \(\mathcal N_{d,\chi}\). A zero on the
critical line of multiplicity \(m\ge2\) contributes \(m\) to
\(\mathcal N_\chi\), zero to \(\mathcal N^s_{0,\chi}\), and one to each
of \(\mathcal N^*_{0,\chi}\) and \(\mathcal N_{d,\chi}\). An off-line
functional-equation orbit of common multiplicity \(m\) contributes \(2m\)
to \(\mathcal N_\chi\), zero to both critical-line counts, and two to
\(\mathcal N_{d,\chi}\). The two orbit points have the same ordinate, so
both lie in the same half-open interval; their distinctness is proved in
Section~8. No conclusion for an individual character follows from the
family ratios.
\end{remark}

\subsection{Auxiliary scales and normalization}
We next fix the auxiliary data in the order used in the proof. Put
\[
 \mathfrak s=1+\frac\eta2>1.
\]
Write \(G_c^{\mathfrak s}((-1/2,1/2))\) for the compactly supported
Gevrey class consisting of those \(\psi\in C_c^\infty((-1/2,1/2))\) for
which there exist constants \(C_\psi,R_\psi>0\) such that
\begin{equation}
\tag{2.5.1}\label{eq:gevrey-definition}
 \|\psi^{(m)}\|_\infty
 \le C_\psi R_\psi^m(m!)^{\mathfrak s}
 \qquad(m\ge0).
\end{equation}
Choose
\begin{equation}
\tag{2.6}\label{eq:2.6}
0<\lambda<1,\qquad 0<\theta<\frac14,\qquad
0\ne\psi\in G_c^{\mathfrak s}\!\left(-\frac12,\frac12\right),
\end{equation}
where \(\psi\) is real and even. The margins \(\eta,A_0\) and all
objects in \eqref{eq:2.6} remain fixed while \(q\to\infty\), whereas
\(T\) may vary within \eqref{eq:2.1}. Set
\begin{equation}
\tag{2.7}\label{eq:2.7}
\begin{aligned}
L&=\lambda Q,\qquad X=e^L=q^\lambda,\\
J&=[(1+\theta)T,(2-\theta)T],\qquad \Delta=|J|=(1-2\theta)T.
\end{aligned}
\end{equation}
Also put
\begin{equation}
\tag{2.8}\label{eq:2.8}
\phi(u)=\psi(u/L),\qquad v=\psi^2,\qquad
a(v)=\int_{-1/2}^{1/2}v(s)\,ds,
\end{equation}
\begin{equation}
\tag{2.9}\label{eq:2.9}
b(v)=\int_{-1/2}^{1/2}v(s)^2\,ds,
\qquad
\mathcal J(v)=\iint_{[-1/2,1/2]^2}
 |s-s'|v(s)v(s')\,ds\,ds',
\end{equation}
and
\begin{equation}
\tag{2.10}\label{eq:2.10}
c_\lambda(v)=\frac{\lambda a(v)^2}{b(v)+\lambda^2\mathcal J(v)}.
\end{equation}
This quotient is the effective signal-to-energy ratio of the window: the
first moment is proportional to \(a(v)\), while the archimedean and prime
parts of the second moment contribute \(b(v)\) and
\(\lambda^2\mathcal J(v)\), respectively. The final simple-zero bound will
contain the penalty \(1/c_\lambda(v)\).

The order of limits at the end of the proof is essential: first
\(q\to\infty\), uniformly over \(T\) in \eqref{eq:2.1}, for fixed
\(\eta,A_0\) and every fixed triple in \eqref{eq:2.6}; then
\(\theta\downarrow0\), then Gevrey approximation to the optimizing window,
and finally the supremum over fixed \(\lambda<1\). We never take
\(\eta=\eta(q)\), \(A_0=A_0(q)\), \(\lambda=\lambda(q)\),
\(\theta=\theta(q)\), or a \(q\)-dependent window.

Our Fourier convention is
\begin{equation}
\tag{2.11}\label{eq:2.11}
\widehat f(t)=\int_{\R}f(u)e^{itu}\,du,
\qquad
f(u)=\frac1{2\pi}\int_{\R}\widehat f(t)e^{-itu}\,dt.
\end{equation}
Thus Parseval's identity contains the factor \((2\pi)^{-1}\). For a
finite matrix \(M\), \(\norm M_F\) and \(\norm M_1\) denote the Frobenius
and trace norms, and \(n_+(M)\) denotes the number of positive eigenvalues of
a Hermitian \(M\), counted with multiplicity. Let
\begin{equation}
\tag{2.12}\label{eq:2.12}
h=\frac{2\pi}{L},\qquad \tau_k=\tau_0+kh,
\qquad \mathcal K=\{k:\tau_k\in J\},
\end{equation}
where the offset \(\tau_0\) is arbitrary and fixed. Then
\begin{equation}
\tag{2.13}\label{eq:2.13}
d:=\#\mathcal K=\frac{\Delta L}{2\pi}+O(1)\asymp TQ=\mathcal H,
\qquad d\gg_{\eta,\lambda,\theta}(\log Q)^{1+\eta}.
\end{equation}
We write
\begin{equation}
\tag{2.14}\label{eq:2.14}
p_k(t)=\widehat\phi(t-\tau_k),\qquad
\Phi=\widehat{\phi^2},\qquad a=a(v).
\end{equation}
All implied constants may depend on \(\eta,A_0\) and the fixed data in
\eqref{eq:2.6}, but are uniform in \(T\) satisfying \eqref{eq:2.1}. When an
arbitrary power \(B>0\) is requested, \(B\) is fixed after the Gevrey
window and before \(q\) tends to infinity.

\subsection{Reduction to the four proof inputs}
The remaining argument is organized around four inputs. Section~3 gives the
denominator asymptotic and the negligible exceptional set; Sections~5--7
give the first and second family moments of the normalized finite matrix;
Section~8 gives the zero-side rank--trace certificate; and Section~9 gives
trace-norm localization for every good character. Once these statements are
available, the proof reduces to the following fixed-parameter assertion.

\begin{proposition}[Fixed auxiliary data: uniform main reduction]
\label{prop:fixed-auxiliary}
For every fixed \(\eta>0\), \(A_0>0\), \(0<\lambda<1\),
\(0<\theta<1/4\), and every fixed real, even, nonzero
\(\psi\in G_c^{\mathfrak s}((-1/2,1/2))\), with \(v=\psi^2\), uniformly for
\(T\) satisfying \eqref{eq:2.1},
\begin{align*}
\frac{\sum_{\chi\ne\chi_0}\mathcal N^s_{0,\chi}(I)}
     {\sum_{\chi\ne\chi_0}\mathcal N_\chi(I)}
&\ge(1-2\theta)\left(4-\frac1{c_\lambda(v)}\right)-2+o_{\eta,A_0}(1),\\
\frac{\sum_{\chi\ne\chi_0}\mathcal N^*_{0,\chi}(I)}
     {\sum_{\chi\ne\chi_0}\mathcal N_\chi(I)}
&\ge(1-2\theta)\left(4-\frac1{c_\lambda(v)}\right)-2+o_{\eta,A_0}(1),\\
\frac{\sum_{\chi\ne\chi_0}\mathcal N_{d,\chi}(I)}
     {\sum_{\chi\ne\chi_0}\mathcal N_\chi(I)}
&\ge\frac12\left\{(1-2\theta)
       \left(4-\frac1{c_\lambda(v)}\right)-1\right\}+o_{\eta,A_0}(1).
\end{align*}
The little-oh terms may depend on the fixed auxiliary data but are uniform in
\(T\).
\end{proposition}

The proof of Proposition~\ref{prop:fixed-auxiliary} is given in Section~10
after the four inputs have been established. Section~11 then optimizes
\(c_\lambda(v)\) and completes the proof of Theorem~\ref{thm:main}.

\section{Counting and zero-density input}

We use the symmetric counting formula, Selberg's family-averaged argument
estimate, and the zero-density estimate of Hiary--Zhao
\cite{Selberg1946,HiaryZhao2026}. Their hypotheses are stated before
specialization.

For \(t>0\), let
\begin{equation}
\tag{3.1}\label{eq:3.1}
N(t,\chi)=\#\{\rho:-t\le\Im\rho\le t\},
\end{equation}
with multiplicity. At a zero ordinate the source defines \(S(t,\chi)\) by
symmetric limiting values. The associated endpoint convention is immaterial
here, because the counting formula is applied only at the shifted
endpoint-free values supplied by Lemma~\ref{lem:half-open-extraction}.
We use the standard parity convention
\(a_\chi=(1-\chi(-1))/2\), consistent with the logarithmic-derivative
formula in Hiary--Zhao~\cite{HiaryZhao2026} and with
DLMF~\cite[(25.15.5)]{DLMF}. In this notation, the argument-principle
identity is
\begin{equation}
\tag{3.2}\label{eq:3.2}
\begin{aligned}
N(t,\chi)
={}&\frac{t}{\pi}\log\frac q\pi+S(t,\chi)+S(t,\bar\chi)\\
&+\frac1{2\pi}\int_{-t}^{t}
 \frac{\Gamma'}{\Gamma}\!\left(\frac14+\frac{a_\chi}{2}+\frac{iu}{2}\right)\,du.
\end{aligned}
\end{equation}
For both parities, the real part of the gamma integrand is
\(O(1+\log(|u|+2))\), uniformly for real \(u\), by the standard digamma
asymptotic.

Selberg's Theorem~8, in the form recalled by Hiary--Zhao, gives for every
fixed \(0<\varepsilon\le1/4\)
\begin{equation}
\tag{3.3}\label{eq:3.3}
\left|\frac1{q-2}\sum_{\chi\ne\chi_0}S(t,\chi)\right|
 \ll_\varepsilon1
\qquad (|t|\le q^{1/4-\varepsilon}).
\end{equation}
The implied constant is independent of \(q\) and \(t\). Hiary--Zhao make
this estimate explicit, with the constant \(1075\), in the smaller range
\(|t|\le1\); the qualitative wider range \eqref{eq:3.3} is what is needed
here. Because the nonprincipal family is closed under
\(\chi\mapsto\bar\chi\), the same average controls both argument terms.

\begin{lemma}[Positive-ordinate family symmetry]\label{lem:positive-symmetry}
For \(0<a<b\), with the same half-open endpoint convention,
\[
\sum_{\chi\ne\chi_0}\#\{\rho:a<\Im\rho\le b\}
 =\frac12\sum_{\chi\ne\chi_0}\#\{\rho:a<|\Im\rho|\le b\}.
\]
Both sides count multiplicity.
\end{lemma}

\begin{proof}
Complex conjugation maps a zero \(\rho\) of \(L(s,\chi)\) to the zero
\(\bar\rho\) of \(L(s,\bar\chi)\), reversing its ordinate and preserving
multiplicity. Since \(\chi\mapsto\bar\chi\) permutes the nonprincipal family,
the family totals for positive and negative ordinates are equal. Explicitly,
the negative component \([-b,-a)\) is carried bijectively, including
multiplicity, to the positive component \((a,b]\).
\end{proof}

\begin{lemma}[Extraction of the half-open interval]
\label{lem:half-open-extraction}
For each prime \(q\), there is a number \(0<\upsilon_q<T/4\) such that no
nontrivial zero of any nonprincipal \(L(s,\chi)\) has absolute ordinate in
\[
 (T,T+\upsilon_q]\cup(2T,2T+\upsilon_q].
\]
For every such choice,
\begin{equation}
\tag{3.3.1}\label{eq:half-open-extraction}
2\mathcal N_q
 =\sum_{\chi\ne\chi_0}
  \{N(2T+\upsilon_q,\chi)-N(T+\upsilon_q,\chi)\}.
\end{equation}
\end{lemma}

\begin{proof}
For fixed \(q\), the union over the finite nonprincipal family of the
nontrivial zeros with \(|\Im\rho|\le3T\) is finite. Hence \(\upsilon_q\) may be
chosen smaller than \(T/4\) and smaller than the distance from each of
\(T\) and \(2T\) to the next larger absolute zero ordinate, if such an
ordinate exists. The difference on the right of
\eqref{eq:half-open-extraction} then counts exactly the zeros with
\(T<|\Im\rho|\le2T\), including zeros at absolute ordinate \(2T\) and
excluding those at absolute ordinate \(T\). Lemma~\ref{lem:positive-symmetry}
shows that this symmetric family count is twice the positive-ordinate count
\(\mathcal N_q\). The shifted endpoints themselves are not zero ordinates,
so no endpoint half-weight occurs.
\end{proof}

Choose \(\upsilon_q\) as in Lemma~\ref{lem:half-open-extraction}. Fix
\(\varepsilon=1/5\) in \eqref{eq:3.3}. Since \(T\le Q^{A_0}\), both shifted endpoints are at most
\(q^{1/20}\) for all sufficiently large \(q\), uniformly in \(T\). We may therefore apply
\eqref{eq:3.2} at \(2T+\upsilon_q\) and \(T+\upsilon_q\), sum over the
nonprincipal family, and use \eqref{eq:half-open-extraction}. The difference
of the two shifted heights is exactly \(T\), so the conductor contribution
is
\[
 \frac{(q-2)T}{\pi}\log\frac q\pi.
\]
At either endpoint, closure under \(\chi\mapsto\bar\chi\) and
\eqref{eq:3.3} show that the difference of the argument contributions is
\(O(q)\). The gamma-integral difference consists of two intervals of total
length \(2T\), on which the integrand is
\(O(1+\log(T+2))\); its family contribution is therefore
\(O(qT\log(T+2))\), uniformly in \(\upsilon_q\). Every count includes
multiplicity. Consequently
\begin{equation}
\tag{3.4}\label{eq:3.4}
\begin{aligned}
2\mathcal N_q
 &=\frac{(q-2)T}{\pi}\log\frac q\pi
   +O(qT\log(T+2))+O(q),\\
\mathcal N_q
 &=\frac{qTQ}{2\pi}+O(qT\log(T+2))+O(q)\\
 &=\frac{qTQ}{2\pi}\{1+o_{\eta,A_0}(1)\},
\end{aligned}
\end{equation}
uniformly for \(T\) in \eqref{eq:2.1}. Indeed the two relative remainders
are
\[
 O\!\left(\frac{\log(T+2)}Q\right)
 +O\!\left(\frac1{TQ}\right)
 \ll_{A_0}\frac{\log Q}{Q}+(\log Q)^{-1-\eta}
 =o_{\eta,A_0}(1).
\]
This proves \eqref{eq:2.5} without requiring either endpoint of \(I\) to be
free of zeros.

We also require the following zero-density estimate. For
\(1/2\le\sigma\le1\), let
\begin{equation}
\tag{3.5}\label{eq:3.5}
N(\sigma;t_1,t_2;\chi)
 =\sum_{\substack{L(\rho,\chi)=0\\ \Re\rho\ge\sigma\\
                  t_1\le\Im\rho\le t_2}}m_\rho.
\end{equation}

\begin{theorem}[Hiary--Zhao zero-density estimate]\label{thm:HZ-density}
Fix \(0<\eps\le1/4\) and \(0<\kappa<\eps/2\).  By
\cite[Theorem~2]{HiaryZhao2026}, for every \(q\ge q_0(\kappa)\), if
\begin{equation}
\tag{3.6}\label{eq:3.6}
\sigma\ge\frac12+\frac5{8\kappa\log q},\qquad
|t_1|,|t_2|\le q^{1/4-\eps},\qquad
t_2-t_1\ge\frac{1.73}{\kappa\log q},
\end{equation}
then
\begin{equation}
\tag{3.7}\label{eq:3.7}
\begin{aligned}
\sum_{\chi\ne\chi_0}N(\sigma;t_1,t_2;\chi)
 &<\left(4.79\kappa+
 \frac{4.12}{2(t_2-t_1)\log q-1.73/\kappa}\right)\\
 &\qquad{}\times q^{1-2\kappa(\sigma-1/2)}(t_2-t_1)\log q.
\end{aligned}
\end{equation}
\end{theorem}

In \eqref{eq:3.7}, the interval-length factor lies outside the power of
\(q\); this is the form used in the shell calculation below.

For later reference we record the consequence of
Theorem~\ref{thm:HZ-density} used to delete exceptional characters. Put
\begin{equation}
\tag{3.8}\label{eq:3.8}
\delta_0=\frac{24\log Q}{Q}.
\end{equation}
Call \(\chi\) \emph{bad} if \(L(s,\chi)\) has a nontrivial zero
\(\rho=1/2+\delta+i\gamma\) satisfying
\begin{equation}
\tag{3.9}\label{eq:3.9}
|\delta|\ge\delta_0,\qquad |\gamma|\le q^{|\delta|/24};
\end{equation}
otherwise call it \emph{good}. The condition is invariant under the
same-character symmetry \(\rho\mapsto1-\bar\rho\).

\begin{lemma}[Exceptional characters]\label{lem:exceptional}
The set \(\mathcal B_q\) of bad nonprincipal characters satisfies
\begin{equation}
\tag{3.10}\label{eq:3.10}
|\mathcal B_q|\ll\frac q{Q^2}.
\end{equation}
\end{lemma}

\begin{proof}
A nontrivial zero has \(|\delta|<1/2\).  If a bad character is witnessed by
a zero with \(\delta\le-\delta_0\), the same-character functional-equation
symmetry
\[
 \frac12+\delta+i\gamma\longmapsto
 \frac12-\delta+i\gamma
\]
produces a witnessing zero in the right half of the critical strip, with the
same ordinate and multiplicity.  Thus it is enough to count witnesses with
\(\delta\in[\delta_0,1/2)\); this reduction introduces no factor two.

Take
\[
 \eps=\frac15,\qquad \kappa=\frac1{12},\qquad
 2\kappa=\frac16,\qquad \frac14-\eps=\frac1{20}.
\]
Put \(\delta_j=\delta_0+j/Q\), and for large \(q\) let
\[
 J=\max\{j\ge0:\delta_j<1/2\}.
\]
For \(0\le j\le J\), define
\[
 \delta_j^+=\min(\delta_{j+1},1/2),\qquad
 \mathfrak S_j=[\delta_j,\delta_j^+),\qquad
 H_j=q^{\delta_j^+/24}.
\]
These shells partition \([\delta_0,1/2)\).  For \(j<J\) their width is
\(1/Q\), while the terminal shell satisfies
\(0<\delta_J^+-\delta_J\le1/Q\).  A bad character represented in
\(\mathfrak S_j\) has, after the preceding symmetry if necessary, a zero
with
\[
 \Re\rho\ge\frac12+\delta_j,
 \qquad |\Im\rho|\le q^{\delta/24}\le H_j.
\]
It therefore contributes at least one, with its full multiplicity, to
\(N(1/2+\delta_j;-H_j,H_j;\chi)\).  Hence the number of characters
represented in the shell is bounded by the corresponding family density
count.

Apply Theorem~\ref{thm:HZ-density} with
\[
 \sigma=\frac12+\delta_j,\qquad t_1=-H_j,\qquad t_2=H_j.
\]
We verify the hypotheses explicitly.  First,
\[
 \delta_j\ge\delta_0=\frac{24\log Q}{Q}
 >\frac5{8\kappa Q}=\frac{15}{2Q}
\]
for large \(q\).  Next,
\[
 H_j\ge q^{\delta_0/24}=Q,
 \qquad
 H_j\le q^{1/48}<q^{1/20}=q^{1/4-\eps}.
\]
Moreover
\[
 t_2-t_1=2H_j\ge2Q>
 \frac{1.73}{\kappa Q}=\frac{20.76}{Q},
\]
and the denominator in the density prefactor is exactly
\[
 2(t_2-t_1)Q-\frac{1.73}{\kappa}
 =4H_jQ-20.76>0.
\]
Thus the full parenthetical factor in \eqref{eq:3.7} is bounded uniformly in
\(j\) and \(q\).  Because the interval-length factor is outside the power of
\(q\), substitution gives
\[
 q^{1-2\kappa\delta_j}(t_2-t_1)Q
 =2q^{1-\delta_j/6}H_jQ.
\]
The shell width, including the terminal one, satisfies
\(0<\delta_j^+-\delta_j\le1/Q\), and therefore
\[
 H_j=q^{\delta_j^+/24}
 \le e^{1/24}q^{\delta_j/24}.
\]
Consequently the number of characters represented in the \(j\)-th shell is
\begin{equation}
\tag{3.11}\label{eq:3.11}
 \ll q^{1-\delta_j/6}H_jQ
 \ll qQq^{-\delta_j/8}.
\end{equation}
Finally,
\[
 q^{-\delta_0/8}=Q^{-3},\qquad
 q^{-(\delta_j-\delta_0)/8}=e^{-j/8}.
\]
Summing over \(0\le j\le J\), including the terminal shell, gives
\[
 |\mathcal B_q|
 \ll qQ\,Q^{-3}\sum_{j=0}^{J}e^{-j/8}
 \ll\frac q{Q^2},
\]
which proves \eqref{eq:3.10}.
\end{proof}

\begin{corollary}[Counting and deletion interface]
\label{cor:counting-deletion-interface}
For every fixed \(\eta>0\) and \(A_0>0\), uniformly for \(T\) satisfying
\eqref{eq:2.1},
\[
\begin{gathered}
 \sum_{\chi\ne\chi_0}\mathcal N_\chi(I)
 =\frac{qTQ}{2\pi}+O(qT\log(T+2))+O(q)
 =\frac{qTQ}{2\pi}\{1+o_{\eta,A_0}(1)\},\\
 |\mathcal B_q|\ll\frac q{Q^2}.
\end{gathered}
\]
Here the first count and all density counts entering the second estimate
include multiplicity.
\end{corollary}

\begin{proof}
The two assertions are \eqref{eq:3.4} and \eqref{eq:3.10}, respectively.
\end{proof}

\begin{lemma}[Uniform local zero count]\label{lem:local-zero-count}
For every primitive nonprincipal \(\chi\pmod q\) and every real \(t\),
\[
\sum_{\substack{L(\rho,\chi)=0\\ |\Im\rho-t|\le1}}m_\rho
 \ll Q+\log(|t|+3).
\]
Consequently, uniformly in \(H\ge0\),
\begin{equation}
\tag{3.12}\label{eq:3.12}
\sum_{\substack{L(\rho,\chi)=0\\ H<|\Im\rho|\le H+1}}m_\rho
 \ll Q+\log(H+3)\qquad(H\ge0),
\end{equation}
where the zeros are nontrivial and counted with multiplicity.
\end{lemma}

\begin{proof}
Let \(N_\chi^{\mathrm{sym}}(U)\) count the nontrivial zeros
\(\rho=\beta+i\gamma\) with \(|\gamma|\le U\), including multiplicity.
For conductor \(q>1\) and \(U\ge5/7\), set
\[
 \ell(U)=\log\frac{q(U+2)}{2\pi}.
\]
The published Theorem~1.1 of Bennett--Martin--O'Bryant--Rechnitzer
\cite{BMOR2021} states that \(N_\chi^{\mathrm{sym}}(U)=0\) when
\(\ell(U)\le1.567\), while for \(\ell(U)>1.567\),
\[
 \left|N_\chi^{\mathrm{sym}}(U)-
 \left\{\frac{U}{\pi}\log\frac{qU}{2\pi e}
       -\frac{\chi(-1)}4\right\}\right|
 \le0.22737\ell(U)+2\log(1+\ell(U))-0.5.
\]
Both branches, together with \(\ell(U)\ll Q+\log(U+3)\), imply
\begin{equation}
\tag{3.12.1}\label{eq:hadamard-log-derivative}
 N_\chi^{\mathrm{sym}}(U)
 =\frac{U}{\pi}\log\frac{qU}{2\pi e}
  -\frac{\chi(-1)}4
  +O\!\left(Q+\log(U+3)\right).
\end{equation}
If \(|t|\le3\), the required unit-window count is at most
\(N_\chi^{\mathrm{sym}}(4)\ll Q\).  Suppose \(|t|>3\). Every zero with
\(|\Im\rho-t|\le1\) belongs to the symmetric shell
\(|t|-2<|\Im\rho|\le |t|+1\); hence
\begin{equation}
\tag{3.13}\label{eq:3.13}
 \sum_{|\Im\rho-t|\le1}m_\rho
 \le N_\chi^{\mathrm{sym}}(|t|+1)
     -N_\chi^{\mathrm{sym}}(|t|-2)
 \ll Q+\log(|t|+3),
\end{equation}
where the last estimate follows by subtracting the two main terms in
\eqref{eq:hadamard-log-derivative} and bounding both explicit remainders.
This proves the first assertion without any symmetry about the real axis for
a fixed nonreal character.  The shell \(H<|\Im\rho|\le H+1\) is covered by
the two unit windows centred at \(H+1/2\) and \(-H-1/2\), which proves
\eqref{eq:3.12}.  All counts include multiplicity and any real exceptional
zero.
\end{proof}

\section{The explicit formula on the matrix test class}

For \(\chi\ne\chi_0\pmod q\), put
\begin{equation}
\tag{4.1}\label{eq:4.1}
a_\chi=\frac{1-\chi(-1)}2\in\{0,1\},
\qquad
\Lambda(s,\chi)=\left(\frac q\pi\right)^{(s+a_\chi)/2}
 \Gamma\!\left(\frac{s+a_\chi}{2}\right)L(s,\chi).
\end{equation}
Since \(q\) is prime, every such character is primitive.  The parity and
completion in \eqref{eq:4.1}, together with the functional equation used
below, agree with DLMF~\cite[(25.15.5)]{DLMF}. For real \(t\), define
\begin{equation}
\tag{4.2}\label{eq:4.2}
\mu_\chi(t)=\frac1{2\pi}\log\frac q\pi
 +\frac1{2\pi}\Re\frac{\Gamma'}{\Gamma}
 \!\left(\frac{1/2+a_\chi+it}{2}\right),
\end{equation}
\begin{equation}
\tag{4.3}\label{eq:4.3}
P_{\chi,L}(t)=-\frac1{2\pi}\sum_{n\le e^L}\frac{\Lambda(n)}{\sqrt n}
 \left(\chi(n)n^{-it}+\overline{\chi(n)}n^{it}\right),
\end{equation}
and
\begin{equation}
\tag{4.4}\label{eq:4.4}
\nu_{\chi,L}(t)=\mu_\chi(t)+P_{\chi,L}(t).
\end{equation}
The object in \eqref{eq:4.4} is an ordinary, smooth, real-valued function
which depends on the bandwidth \(L\). It is not a pointwise density of zeros.
We use the standard Weil explicit formula.  Zhao
\cite[Lemma~4]{Zhao2026} records the same zero--gamma--prime structure for
strip-analytic tests that are real on the real axis, and cites
\cite[Theorem~5.12]{IwaniecKowalski2004}.  We fix the standard parity by
DLMF~\cite[(25.15.5)]{DLMF}; the Fourier conversion, the complex-linear
extension, and every convention-dependent factor are displayed explicitly
below.

For \(k\in\mathcal K\), let
\begin{equation}
\tag{4.5}\label{eq:4.5}
f_k(u)=\phi(u)e^{-i\tau_ku},
\qquad p_k(t)=\widehat\phi(t-\tau_k).
\end{equation}

We first isolate the scalar formula used for the matrix entries, making the
Fourier and polarization conventions explicit for complex, non-even tests.

\begin{proposition}[Complex Weil formula on the Paley--Wiener class]
\label{prop:complex-weil}
Let \(\widetilde H\in C_c^\infty(\R)\), and define
\[
 H(z)=\int_{\R}\widetilde H(u)e^{izu}\,du,
 \qquad
 \widetilde H(u)=\frac1{2\pi}\int_{\R}H(t)e^{-itu}\,dt.
\]
Then, for every primitive nonprincipal \(\chi\pmod q\),
\begin{equation}
\tag{4.5.1}\label{eq:complex-weil}
\begin{aligned}
\sum_{\substack{L(\rho,\chi)=0\\ \rho\ \mathrm{nontrivial}}}m_\rho
 H\!\left(\frac{\rho-1/2}{i}\right)
={}&\int_{\R}H(t)\mu_\chi(t)\,dt\\
&-\sum_{n\ge2}\frac{\Lambda(n)}{\sqrt n}
 \Bigl\{\chi(n)\widetilde H(\log n)\\
&\hspace{8em}{}+\overline{\chi(n)}\widetilde H(-\log n)\Bigr\}.
\end{aligned}
\end{equation}
The zero sum and the integral are absolutely convergent, and the prime sum
is finite.  No parity or reality condition on \(H|_{\R}\) is required.
\end{proposition}

\begin{proof}
The Paley--Wiener representation makes \(H\) entire and rapidly decreasing
in the real direction in every fixed horizontal strip.  Since the
nontrivial zeros satisfy \(0<\Re\rho<1\), Lemma~\ref{lem:local-zero-count}
therefore gives absolute convergence of the zero sum.

For a function satisfying \(H^\sharp=H\), where
\(H^\sharp(z)=\overline{H(\overline z)}\), formula
\eqref{eq:complex-weil} is Zhao's real-boundary explicit formula
\cite[Lemma~4]{Zhao2026} after translating from the convention
\(\widehat f_Z(\xi)=\int f(t)e^{-2\pi i\xi t}\,dt\): namely,
\(\widehat f_Z(\log n/(2\pi))=2\pi\widetilde H(\log n)\).
We also record the contour normalization.  Apply the
residue theorem to the logarithmic derivative
of the completed function \eqref{eq:4.1}, with the factor
\(H((s-1/2)/i)\), on a symmetric rectangle whose horizontal sides avoid the
zeros.  The horizontal integrals vanish by the strip decay of \(H\).  On the
right vertical side, the absolutely convergent expansion
\[
 -\frac{L'}{L}(s,\chi)
 =\sum_{n\ge2}\frac{\Lambda(n)\chi(n)}{n^s}
\]
and Fourier inversion produce
\(\chi(n)\widetilde H(\log n)/\sqrt n\).  On the opposite side, the
functional equation
\(\Lambda(s,\chi)=\varepsilon_\chi\Lambda(1-s,\overline\chi)\), followed by
\(s\mapsto1-s\), produces
\(\overline{\chi(n)}\widetilde H(-\log n)/\sqrt n\).  The conductor and
gamma terms combine to
\[
 \frac1{2\pi}\int_{\R}H(t)
 \left\{\log\frac q\pi+
 \Re\frac{\Gamma'}{\Gamma}
 \!\left(\frac{1/2+a_\chi+it}{2}\right)\right\}\,dt,
\]
which is the integral against \(\mu_\chi\) in \eqref{eq:4.2}.  The residues
of the completed logarithmic derivative are precisely the nontrivial zeros,
with multiplicity.  This proves \eqref{eq:complex-weil} for real-valued
boundary data in the stated Fourier convention; it is the same contour
formula as \cite[Theorem~5.12]{IwaniecKowalski2004} and
\cite[Lemma~4]{Zhao2026}.  The parity in the gamma factor is the standard
one fixed independently by \eqref{eq:4.1} and DLMF~(25.15.5).

For general \(H\), put
\[
 H_1=\frac{H+H^\sharp}{2},\qquad
 H_2=\frac{H-H^\sharp}{2i}.
\]
Both functions belong to the same Paley--Wiener class and are real-valued on
\(\R\).  Apply the preceding formula to \(H_1\) and \(H_2\), then recombine
by complex linearity.  This proves the proposition.
\end{proof}

For the present test functions, define the polarized Weil form by
\[
 W_\chi(f_k,f_\ell)
 :=\sum_{\substack{L(\rho,\chi)=0\\ \rho\ \mathrm{nontrivial}}}
 m_\rho\,\widehat f_k(z_\rho)
 \overline{\widehat f_\ell(\overline{z_\rho})},
 \qquad z_\rho=\frac{\rho-1/2}{i}.
\]
The absolute convergence established above makes this definition independent
of any ordering of the zeros. Define
\(G_{\chi,k\ell}=W_\chi(f_k,f_\ell)\). Its zero side is therefore
\begin{equation}
\tag{4.6}\label{eq:4.6}
G_{\chi,k\ell}
 =\sum_{\substack{L(\rho,\chi)=0\\ \rho\ \mathrm{nontrivial}}}
 m_\rho\,p_k(z_\rho)\overline{p_\ell(\bar z_\rho)},
\qquad z_\rho=\frac{\rho-1/2}{i}.
\end{equation}
Zeros are counted with multiplicity.  Formula \eqref{eq:4.6} also fixes the
complex evaluation convention: for an off-line zero
\(\rho=1/2+\delta+i\gamma\), one has \(z_\rho=\gamma-i\delta\), not merely
\(\gamma\).  The series is absolutely convergent by
Lemma~\ref{lem:local-zero-count} and Paley--Wiener decay in the fixed strip
\(|\Im z|<1/2\). Section~9 strengthens this to trace-norm convergence,
which justifies grouping into functional-equation orbits and splitting the
matrix into its interior and exterior pieces.

\begin{lemma}[Test-class representative]\label{lem:test-class-rep}
For every \(k,\ell\in\mathcal K\),
\begin{equation}
\tag{4.7}\label{eq:4.7}
\boxed{
G_{\chi,k\ell}=\int_{\R}p_k(t)\overline{p_\ell(t)}\nu_{\chi,L}(t)\,dt.}
\end{equation}
The integral is absolutely convergent. Moreover,
\(G_{\chi,\ell k}=\overline{G_{\chi,k\ell}}\). Uniformly for all real \(t\),
\begin{equation}
\tag{4.8}\label{eq:4.8}
\sum_{\chi\ne\chi_0}|\nu_{\chi,L}(t)|^2
 \ll q\{Q+\log(|t|+2)\}^2.
\end{equation}
\end{lemma}

\begin{proof}
Set
\begin{equation}
\tag{4.9}\label{eq:4.9}
H_{k\ell}(t)=p_k(t)\overline{p_\ell(t)},
\qquad
\widetilde H_{k\ell}(u)=\frac1{2\pi}\int_{\R}H_{k\ell}(t)e^{-itu}\,dt.
\end{equation}
Because \(\phi\) is real, the entire extension of the first function is
\begin{equation}
\tag{4.10}\label{eq:4.10}
H_{k\ell}(z)=\widehat\phi(z-\tau_k)
 \overline{\widehat\phi(\bar z-\tau_\ell)}.
\end{equation}
The product--convolution formula in the convention \eqref{eq:2.11} gives
\begin{equation}
\tag{4.11}\label{eq:4.11}
\widetilde H_{k\ell}(u)=\int_{\R}
 \phi(y+u)\phi(y)e^{-i\tau_k(y+u)}e^{i\tau_\ell y}\,dy.
\end{equation}
Both \(y\) and \(y+u\) must lie in the support of \(\phi\). Hence
\begin{equation}
\tag{4.12}\label{eq:4.12}
\supp\widetilde H_{k\ell}
 \subset\supp\phi-\supp\phi\Subset(-L,L).
\end{equation}
In particular, \(\widetilde H_{k\ell}(\pm L)=0\).

Apply Proposition~\ref{prop:complex-weil} to \(H=H_{k\ell}\).  Equations
\eqref{eq:4.11}--\eqref{eq:4.12} verify its exact test-class hypotheses.  In
particular,
\begin{equation}
\tag{4.15}\label{eq:4.15}
\int H_{k\ell}(t)n^{-it}\,dt=2\pi\widetilde H_{k\ell}(\log n),
\qquad
\int H_{k\ell}(t)n^{it}\,dt=2\pi\widetilde H_{k\ell}(-\log n).
\end{equation}
By \eqref{eq:4.12}, the \(n\)-th prime coefficient in
\eqref{eq:complex-weil} vanishes when \(\log n\ge L\).  Thus the prime sum
is exactly \(\int H_{k\ell}(t)P_{\chi,L}(t)\,dt\).  If \(e^L\) is an
integer, its endpoint coefficient is zero, so the notation \(n\le e^L\)
creates no ambiguity.  Proposition~\ref{prop:complex-weil} now proves
\eqref{eq:4.7}.

The two terms in parentheses in \eqref{eq:4.3} are conjugate, so
\(P_{\chi,L}\), and hence \(\nu_{\chi,L}\), is real. Also
\(H_{\ell k}=\overline{H_{k\ell}}\) on the real axis, which proves
Hermiticity. Absolute convergence follows because \(H_{k\ell}\) is Schwartz,
\(P_{\chi,L}\) is a finite trigonometric polynomial, and the uniform
digamma asymptotic DLMF~\cite[(5.11.2)]{DLMF} gives
\begin{equation}
\tag{4.16}\label{eq:4.16}
|\mu_\chi(t)|\ll Q+\log(|t|+2).
\end{equation}
It remains to prove \eqref{eq:4.8}. Write
\[
D_\chi(t)=\sum_{n\le X}\frac{\Lambda(n)}{\sqrt n}\chi(n)n^{-it}.
\]
Then \(P_{\chi,L}(t)=-\pi^{-1}\Re D_\chi(t)\). Since the fixed inequality
\(\lambda<1\) gives \(X<q\), character orthogonality is exactly diagonal:
\begin{equation}
\tag{4.17}\label{eq:4.17}
\sum_{\chi\bmod q}|D_\chi(t)|^2
 =(q-1)\sum_{n\le X}\frac{\Lambda(n)^2}{n}\ll qL^2,
\end{equation}
uniformly in real \(t\). Combining \eqref{eq:4.16}--\eqref{eq:4.17} with
\(|x+y|^2\le2|x|^2+2|y|^2\) proves \eqref{eq:4.8}.
\end{proof}

For later parity bookkeeping, if \(\mu_+\) and \(\mu_-\) denote
\eqref{eq:4.2} for even and odd characters, respectively, then
\begin{equation}
\tag{4.18}\label{eq:4.18}
A_q=\frac{\mu_++\mu_-}{2},\qquad
B=\frac{\mu_+-\mu_-}{2},\qquad
\mu_\chi=A_q+\chi(-1)B.
\end{equation}
The functions \(A_q,B\) are real and even. Uniformly for real \(t\),
\[
 A_q(t)=\frac{Q}{2\pi}+O\!\left(1+\log(|t|+2)\right),
 \qquad B(t)=O(1).
\]
Moreover
\begin{equation}
\tag{4.19}\label{eq:4.19}
\nu_{\bar\chi,L}(t)=\nu_{\chi,L}(-t).
\end{equation}
These identities will be used only in full-family character sums. No
primitive conductor-\(q\) explicit formula will be applied to the principal
character; any principal term introduced later is an explicitly labelled
algebraic bookkeeping device.

\section{Gevrey Gabor compression}

Throughout this and the following two sections, fix \(\eta>0\),
\(A_0>0\), and
\[
0<\lambda<1,\qquad 0<\theta<\frac14,\qquad
0\ne\psi\in G_c^{\mathfrak s}\!\left(-\frac12,\frac12\right),
\qquad \mathfrak s=1+\frac\eta2,
\]
where \(\psi\) is real and even. Let \(q\) tend to infinity through primes,
and let \(T\) vary subject to \eqref{eq:2.1}. Put
\[
Q=\log q,\quad L=\lambda Q,\quad X=e^L=q^\lambda,
\]
\[
I=(T,2T],\quad J=[c_-,c_+]=[(1+\theta)T,(2-\theta)T],
\quad \Delta=|J|=(1-2\theta)T.
\]
All implied constants may depend on \(\eta,A_0,\lambda,\theta,\psi\), but are
uniform over the admissible choices of \(T\).
Our Fourier convention is
\begin{equation}
\tag{5.1}\label{eq:5.1}
\widehat f(t)=\int_{\R}f(u)e^{itu}\,du,
\qquad
f(u)=\frac1{2\pi}\int_{\R}\widehat f(t)e^{-itu}\,dt.
\end{equation}
Set
\[
\phi(u)=\psi(u/L),\qquad v(s)=\psi(s)^2,
\]
and define
\begin{equation}
\tag{5.2}\label{eq:5.2}
\begin{aligned}
a&=\int_{-1/2}^{1/2}v(s)\,ds,
& b&=\int_{-1/2}^{1/2}v(s)^2\,ds,\\
\mathcal J(v)&=\iint_{[-1/2,1/2]^2}|s-s'|v(s)v(s')\,ds\,ds'.
\end{aligned}
\end{equation}
We write \(\Phi=\widehat{\phi^2}\). Since
\(\Phi(r)=L\widehat v(Lr)\), Parseval and a change of variables give
\begin{equation}
\tag{5.3}\label{eq:5.3}
\int_{\R}\Phi(r)^2\,dr=2\pi bL,
\qquad
\int_{\R}|r|\Phi(r)^2\,dr
 =\int_{\R}|x|\,|\widehat v(x)|^2\,dx=O_v(1).
\end{equation}
Choose a translated lattice
\[
\tau_k=\tau_0+\frac{2\pi k}{L},\qquad k\in\Z,
\]
and retain the centres in
\begin{equation}
\tag{5.4}\label{eq:5.4}
\begin{aligned}
\mathcal K&=\{k:\tau_k\in J\},\qquad
 d=|\mathcal K|=\frac{\Delta L}{2\pi}+O(1),\\
 d&\asymp TQ=\mathcal H,\qquad
 d\gg_{\eta,\lambda,\theta}(\log Q)^{1+\eta}.
\end{aligned}
\end{equation}
Put
\begin{equation}
\tag{5.5}\label{eq:5.5}
p_k(t)=\widehat\phi(t-\tau_k),
\qquad
K_d(t,t')=\sum_{k\in\mathcal K}p_k(t)\overline{p_k(t')}.
\end{equation}

\subsection{Poisson summation and finite-centre estimates}

\begin{lemma}[Poisson--Gabor identity]\label{lem:poisson-gabor}
For real \(t,t'\),
\begin{equation}
\tag{5.6}\label{eq:5.6}
\sum_{k\in\Z}p_k(t)\overline{p_k(t')}=L\Phi(t-t').
\end{equation}
In particular,
\begin{equation}
\tag{5.7}\label{eq:5.7}
\sum_{k\in\Z}|p_k(t)|^2=aL^2.
\end{equation}
For complex \(z,z'\), polarization gives the right-hand side
\(L\widehat{|\phi|^2}(z-\overline{z'})\).
\end{lemma}

\begin{proof}
Set
\(\Upsilon(x)=\widehat\phi(t-x)\overline{\widehat\phi(t'-x)}\).
This is a Schwartz function.  Apply the scaled and translated Poisson
summation formula DLMF~\cite[(1.8.14)]{DLMF} to the lattice
\(\tau_0+(2\pi/L)\Z\).  A direct Fourier calculation shows that the
transform of \(\Upsilon\) is supported in
\(\supp\phi-\supp\phi\Subset(-L,L)\).  The nonzero dual frequencies are
integer multiples of \(L\), so only the zero mode survives.  Plancherel in
the convention \eqref{eq:5.1} gives
\[
 \int_{\R}\Upsilon(x)\,dx
 =2\pi\int_{\R}|\phi(u)|^2e^{i(t-t')u}\,du
 =2\pi\Phi(t-t').
\]
Multiplication by the lattice density \(L/(2\pi)\) proves
\eqref{eq:5.6}.  Putting \(t=t'\) gives \eqref{eq:5.7}. Because \(\phi\)
has compact support, both sides are entire in \(z\) and in the independent variable
\(w=\overline{z'}\). The lattice series, together with all of its
derivatives, converges normally on compact subsets by Schwartz decay in the
lattice index. The real identity therefore extends first in \(z\) and then
in \(w\) by the one-variable identity theorem. Substituting
\(w=\overline{z'}\) proves the asserted complex polarization.
\end{proof}

For every \(A>0\), repeated integration by parts gives the envelope
\begin{equation}
\tag{5.8}\label{eq:5.8}
|p_k(t)|\le\omega_A(t-\tau_k),
\qquad
\omega_A(x)=C_A L(1+L|x|)^{-A}.
\end{equation}
The finite-centre replacements below use the envelope \eqref{eq:5.8}
directly.  This keeps the proof focused on the endpoint sums that actually
enter the first and second moments.

\subsection{The ordinary explicit-formula representative}

For a nonprincipal character \(\chi\pmod q\), let
\(a_\chi=(1-\chi(-1))/2\), and set
\begin{equation}
\tag{5.9}\label{eq:5.9}
\mu_\chi(t)=\frac1{2\pi}\log\frac q\pi
 +\frac1{2\pi}\Re\frac{\Gamma'}{\Gamma}
 \!\left(\frac{1/2+a_\chi+it}{2}\right),
\end{equation}
\begin{equation}
\tag{5.10}\label{eq:5.10}
P_{\chi,L}(t)=-\frac1{2\pi}\sum_{n\le X}\frac{\Lambda(n)}{\sqrt n}
 \bigl(\chi(n)n^{-it}+\overline{\chi(n)}n^{it}\bigr),
\qquad
\nu_{\chi,L}=\mu_\chi+P_{\chi,L}.
\end{equation}
This is an ordinary, bandwidth-dependent function, not a pointwise zero
density. The inverse transform of \(p_k\overline{p_\ell}\) is compactly
supported in \((-L,L)\). The Weil explicit formula therefore yields exactly
\begin{equation}
\tag{5.11}\label{eq:5.11}
G_{\chi,k\ell}=\int_{\R}p_k(t)\overline{p_\ell(t)}
 \nu_{\chi,L}(t)\,dt.
\end{equation}
All integrals are absolutely convergent. Because \(X<q\), character
orthogonality gives, uniformly for real \(t\),
\begin{equation}
\tag{5.12}\label{eq:5.12}
F(t)^2:=\sum_{\chi\ne\chi_0}|\nu_{\chi,L}(t)|^2
 \ll q\{Q+\log(2+|t|)\}^2.
\end{equation}
Indeed the prime part follows from
\begin{equation}
\tag{5.13}\label{eq:5.13}
\sum_{\chi\bmod q}
 \left|\sum_{n\le X}\frac{\Lambda(n)}{\sqrt n}\chi(n)n^{-it}\right|^2
 =(q-1)\sum_{n\le X}\frac{\Lambda(n)^2}{n}\ll qL^2,
\end{equation}
and the gamma part follows from the uniform digamma asymptotic
DLMF~\cite[(5.11.2)]{DLMF}. We shall also use
\begin{equation}
\tag{5.14}\label{eq:5.14}
\sum_{\chi\ne\chi_0}|\nu_{\chi,L}(t)|
 \le\sqrt q\,F(t)\ll q\{Q+\log(2+|t|)\},
\end{equation}
and, crucially for signed products,
\begin{equation}
\tag{5.15}\label{eq:5.15}
\sum_{\chi\ne\chi_0}
 |\nu_{\chi,L}(t)\nu_{\chi,L}(t')|
 \le F(t)F(t').
\end{equation}
We normalize the matrix by
\begin{equation}
\tag{5.16}\label{eq:5.16}
\widehat G_\chi=\frac{G_\chi}{aL^2}.
\end{equation}
The factor \(aL^2\) is the complete-lattice squared norm in
\eqref{eq:5.7}.  With this normalization, an on-line zero contributes a
positive rank-one block of trace at most one.  The first moment of
\(\widehat G_\chi\) measures total zero-side mass, while its Frobenius second
moment is the concentration penalty that enters the rank--trace certificate.

\subsection{Finite-centre replacement}

\begin{proposition}[Finite-centre first moment]\label{prop:finite-first}
One has
\begin{equation}
\tag{5.17}\label{eq:5.17}
\sum_{\chi\ne\chi_0}\tr\widehat G_\chi
 =\sum_{\chi\ne\chi_0}\int_J\nu_{\chi,L}(t)\,dt+O(q).
\end{equation}
In particular, the error is \(o_{\eta,A_0}(qTQ)\), uniformly in \(T\).
\end{proposition}

\begin{proof}
On \(J\), subtract the finite sum in \eqref{eq:5.7} from the infinite sum;
on \(J^c\), retain the finite sum. By \eqref{eq:5.14}, the absolute error
in \eqref{eq:5.17} is at most
\begin{equation}
\tag{5.18}\label{eq:5.18}
\frac{Cq}{aL^2}
\left\{
 \sum_{k\notin\mathcal K}\int_J\omega_A(t-\tau_k)^2W(t)\,dt
 +\sum_{k\in\mathcal K}\int_{J^c}\omega_A(t-\tau_k)^2W(t)\,dt
\right\},
\end{equation}
where \(W(t)=Q+\log(2+|t|)\). In the first integral \(t\in J\), so
\(W(t)\ll_{A_0} Q\), independently of the omitted centre. In the second
integral \(k\in\mathcal K\), whence \(|\tau_k|\ll T\); after writing
\(x=L(t-\tau_k)\), for every fixed \(\eps>0\),
\begin{equation}
\tag{5.19}\label{eq:5.19}
W(\tau_k+x/L)\ll_{A_0,\eps} Q(1+|x|)^\eps.
\end{equation}
In either integral each squared envelope supplies the factor \(L\). The
centres crossing either endpoint have scaled depths separated by \(2\pi\).
Thus, for \(A>2\), the sum of all resulting tails is \(O(1)\), and the
braces in \eqref{eq:5.18} are \(O(QL)\). Therefore \eqref{eq:5.18} is
\(O(qQ/L)=O(q)\).
\end{proof}

\begin{lemma}[Discrete endpoint tails]\label{lem:discrete-tails}
Let \(F(t)^2=\sum_{\chi\ne\chi_0}|\nu_{\chi,L}(t)|^2\), and for a
measurable set \(U\subset\R\) put
\[
A_k(U)=\int_U\omega_A(t-\tau_k)F(t)\,dt.
\]
If \(A>5/2\), then
\[
\sum_{k\notin\mathcal K}A_k(J)^2\ll qQ^2,
\qquad
\sum_{k\in\mathcal K}A_k(J)A_k(J^c)\ll qQ^2,
\qquad
\sum_{k\in\mathcal K}A_k(J^c)^2\ll qQ^2.
\]
The constants are uniform in the lattice offset \(\tau_0\).
\end{lemma}

\begin{proof}
By \eqref{eq:5.12},
\(F(t)\ll\sqrt q\{Q+\log(2+|t|)\}\). If \(k\notin\mathcal K\), then
\(t\in J\) implies
\(|L(t-\tau_k)|\ge D_k:=L\dist(\tau_k,J)\), and
\(Q+\log(2+|t|)\ll_{A_0} Q\). The change of variables
\(x=L(t-\tau_k)\) therefore gives
\[
A_k(J)\ll\sqrt q\,Q(1+D_k)^{-A+1}.
\]
For \(k\in\mathcal K\), one has \(|\tau_k|\ll T\); the same substitution
over the whole line, using the integrability of
\((1+|x|)^{-A}\log(2+|x|)\), gives
\[
A_k(\R)\ll\sqrt q\,Q.
\]
On \(J^c\), put \(E_k=L\dist(\tau_k,J^c)\). The elementary bound
\[
Q+\log(2+|\tau_k+x/L|)\ll Q(1+|x|)^{1/2}
\]
gives
\[
A_k(J^c)\ll\sqrt q\,Q(1+E_k)^{-A+3/2}.
\]
In scaled coordinates the centres have spacing \(2\pi\). On either side of
each endpoint, \(D_k\) and \(E_k\) therefore dominate sequences of the form
\(2\pi j+O(1)\), with at most two terms at every depth. Hence, for every
\(s>1\),
\[
\sum_{k\notin\mathcal K}(1+D_k)^{-s}
 +\sum_{k\in\mathcal K}(1+E_k)^{-s}\ll_s1.
\]
Squaring the first bound, multiplying the second by the third, and squaring
the third, respectively, reduces the three asserted estimates to the last
display with exponents \(2A-2\), \(A-3/2\), and \(2A-3\). They all
exceed one when \(A>5/2\).
\end{proof}

\begin{proposition}[Finite-centre second moment]\label{prop:finite-second}
One has
\begin{equation}
\tag{5.20}\label{eq:5.20}
\sum_{\chi\ne\chi_0}\tr\bigl((G_\chi/L)^2\bigr)
 =\sum_{\chi\ne\chi_0}\iint_{J^2}\Phi(t-t')^2
  \nu_{\chi,L}(t)\nu_{\chi,L}(t')\,dt\,dt'
 +O(qQ^2).
\end{equation}
Consequently the normalized error, after division by \(a^2L^2\), is
\(O(q)=o_{\eta,A_0}(qTQ)\), uniformly in \(T\).
\end{proposition}

\begin{proof}
Fix a Schwartz exponent \(A>5/2\). Let \(K_\infty=L\Phi\) and
\(K_{\mathrm{out}}=K_\infty-K_d\). From \eqref{eq:5.7} and
Cauchy--Schwarz,
\[
\tr\bigl((G_\chi/L)^2\bigr)
 =\frac1{L^2}\iint_{\R^2}|K_d(t,t')|^2
   \nu_{\chi,L}(t)\nu_{\chi,L}(t')\,dt\,dt'.
\]
Indeed, expand the trace as
\(L^{-2}\sum_{k,\ell}G_{\chi,k\ell}G_{\chi,\ell k}\) and use
finite-dimensional Fubini; absolute convergence of every entry was proved
above. Moreover,
\begin{equation}
\tag{5.21}\label{eq:5.21}
|K_d(t,t')|+|K_\infty(t,t')|\le2aL^2.
\end{equation}
On \(J^2\),
\begin{equation}
\tag{5.22}\label{eq:5.22}
\frac{\bigl||K_d|^2-|K_\infty|^2\bigr|}{L^2}
 \le2a|K_{\mathrm{out}}|
 \le2a\sum_{k\notin\mathcal K}
 \omega_A(t-\tau_k)\omega_A(t'-\tau_k),
\end{equation}
whereas on \((J^2)^c\), where the comparison kernel is zero,
\begin{equation}
\tag{5.23}\label{eq:5.23}
\frac{|K_d|^2}{L^2}
 \le a\sum_{k\in\mathcal K}
 \omega_A(t-\tau_k)\omega_A(t'-\tau_k).
\end{equation}
Define \(A_k(U)=\int_U\omega_A(t-\tau_k)F(t)\,dt\). After applying
\eqref{eq:5.15}, the total absolute error is bounded by
\begin{equation}
\tag{5.24}\label{eq:5.24}
C\sum_{k\notin\mathcal K}A_k(J)^2
 +C\sum_{k\in\mathcal K}
 \{2A_k(J)A_k(J^c)+A_k(J^c)^2\}.
\end{equation}
Lemma~\ref{lem:discrete-tails} shows term by term that the three sums in
\eqref{eq:5.24} are \(O(qQ^2)\).

Finally, finite-dimensional Fubini applied to \eqref{eq:5.11} identifies
the original double integral with \(\sum_\chi\tr((G_\chi/L)^2)\),
proving \eqref{eq:5.20}.
\end{proof}

\section{The prime-side first moment}

Let
\begin{equation}
\tag{6.1}\label{eq:6.1}
\mathcal N_q=\sum_{\chi\ne\chi_0}\mathcal N_\chi(I)
 =\frac{qTQ}{2\pi}\{1+o_{\eta,A_0}(1)\}.
\end{equation}

\begin{proposition}[First moment]\label{prop:first-moment}
For the fixed data above, uniformly for \(T\) satisfying \eqref{eq:2.1},
\begin{equation}
\tag{6.2}\label{eq:6.2}
\sum_{\chi\ne\chi_0}\tr\widehat G_\chi
 =(1-2\theta)\mathcal N_q+o_{\eta,A_0}(\mathcal N_q).
\end{equation}
\end{proposition}

\begin{proof}
Uniformly on \(J\), the digamma asymptotic
DLMF~\cite[(5.11.2)]{DLMF} gives
\(\mu_\chi(t)=Q/(2\pi)+O_{A_0}(\log Q)\), and hence
\begin{equation}
\tag{6.3}\label{eq:6.3}
\sum_{\chi\ne\chi_0}\int_J\mu_\chi(t)\,dt
 =\frac{q\Delta Q}{2\pi}+O_{A_0}(q\Delta\log Q+\Delta Q).
\end{equation}
For every \(1<n\le X<q\), full-family orthogonality gives
\(\sum_{\chi\bmod q}\chi(n)=0\), and likewise for
\(\overline{\chi(n)}\). Thus the full-family prime contribution vanishes.
Removing the principal term leaves the correction polynomial
\begin{equation}
\tag{6.4}\label{eq:6.4}
P_0^*(t)=-\frac1\pi\Re\sum_{n\le X}\frac{\Lambda(n)}{\sqrt n}n^{-it},
\qquad
|P_0^*(t)|\ll\sqrt X.
\end{equation}
The bound in \eqref{eq:6.4} follows by the triangle inequality and the
second estimate in Lemma~\ref{lem:prime-square-sums}; in particular it is
uniform in real \(t\). Its integral over \(J\) is
\(O(\Delta\sqrt X)=o_{\eta,A_0}(qTQ)\), uniformly in \(T\), since \(\lambda<1\) is fixed.
The error in Proposition~\ref{prop:finite-first} is relatively
\(O((TQ)^{-1})\), while the two errors in \eqref{eq:6.3} are relatively
\(O_{A_0}(\log Q/Q)+O(q^{-1})\). Combining these estimates with
\(\Delta=(1-2\theta)T\) and \eqref{eq:6.1} proves \eqref{eq:6.2}. No primitive
conductor-\(q\) explicit formula is applied to the principal character.
\end{proof}

\section{The prime-side second moment}

We isolate the only prime-number-theorem input used in the moment calculation.

\begin{lemma}[Prime-square sums]\label{lem:prime-square-sums}
Uniformly for \(y\ge3\),
\[
\sum_{n\le y}\frac{\Lambda(n)^2}{n}
 =\frac12(\log y)^2+O(\log y),
\qquad
\sum_{n\le y}\frac{\Lambda(n)}{\sqrt n}\ll\sqrt y.
\]
In particular, the first sum is \(\ll(\log y)^2\).
\end{lemma}

\begin{proof}
Johnston--Yang \cite[Theorem~1.1 and Corollary~1.2]{JohnstonYang2023}
give explicit exponentially decaying errors for \(\psi(x)-x\) and
\(\vartheta(x)-x\).  In particular, after enlarging constants on a bounded
range,
\(\vartheta(x)=x+O_A(x(\log x)^{-A})\) for every fixed \(A>0\), and
\(\psi(x)\ll x\) for \(x\ge2\).  Stieltjes partial summation with
\(d\vartheta(t)\) gives
\[
\sum_{p\le y}\frac{(\log p)^2}{p}
 =\int_{2^-}^y\frac{\log t}{t}\,d\vartheta(t)
 =\frac12(\log y)^2+O(\log y).
\]
The terms \(n=p^k\), \(k\ge2\), contribute \(O(1)\), since
\(\sum_p(\log p)^2/(p(p-1))<\infty\). This proves the first assertion. For
the second, Chebyshev's bound \(\psi(x)\ll x\) and partial summation give
\[
\sum_{n\le y}\frac{\Lambda(n)}{\sqrt n}
 =\frac{\psi(y)}{\sqrt y}
  +\frac12\int_1^y\frac{\psi(t)}{t^{3/2}}\,dt
 \ll\sqrt y.
\]
\end{proof}

For real functions \(U,V\), abbreviate
\begin{equation}
\tag{7.1}\label{eq:7.1}
\mathcal M[U,V]
 =\iint_{J^2}\Phi(t-t')^2U(t)V(t')\,dt\,dt'.
\end{equation}

\subsection{The archimedean contribution}

Uniformly on \(J\), one has
\(\mu_\chi(t)=Q/(2\pi)+O_{A_0}(\log Q)\). Hence \eqref{eq:5.3} and the
overlap identity give
\begin{equation}
\tag{7.2}\label{eq:7.2}
\sum_{\chi\ne\chi_0}\mathcal M[\mu_\chi,\mu_\chi]
 =\frac{q\Delta bLQ^2}{2\pi}
  +O_\psi(qQ^2)+O_{\eta,\psi}(q\Delta LQ\log Q).
\end{equation}
Indeed
\(\iint_{J^2}\Phi(t-t')^2\,dt\,dt'=2\pi bL\Delta+O_\psi(1)\), because
\(\int|r|\Phi(r)^2\,dr=O_\psi(1)\); this endpoint error contributes
\(O(qQ^2)\). Replacing one of the two constant gamma factors by its
\(O_{A_0}(\log Q)\) remainder costs
\(O_{\eta,\psi}(q\Delta LQ\log Q)\), and the product of two remainders is smaller.
After normalization, the two displayed errors are relatively
\(O((TQ)^{-1})\) and \(O_{A_0}(\log Q/Q)\), respectively.

\subsection{Expansion of the four prime products}

Write \(a_n=\Lambda(n)/\sqrt n\). The product
\(P_{\chi,L}(t)P_{\chi,L}(t')\) contains the following four terms:
\begin{equation}
\tag{7.3}\label{eq:7.3}
\begin{array}{c@{\qquad}c@{\qquad}c}
\text{character product}&\text{phase}&\text{full-family condition}\\[2pt]
\chi(n)\overline{\chi(m)}&n^{-it}m^{it'}&n\equiv m\pmod q\\
\overline{\chi(n)}\chi(m)&n^{it}m^{-it'}&n\equiv m\pmod q\\
\chi(n)\chi(m)&n^{-it}m^{-it'}&nm\equiv1\pmod q\\
\overline{\chi(n)}\overline{\chi(m)}&n^{it}m^{it'}&nm\equiv1\pmod q.
\end{array}
\end{equation}
Since \(n,m\le X<q\), the first two congruences force \(n=m\).

For one ratio term, writing \(r=t-t'\) gives, up to the endpoint overlap,
\begin{equation}
\tag{7.4}\label{eq:7.4}
\frac{q-1}{4\pi^2}\Delta\sum_{n\le X}a_n^2
 \int_{\R}\Phi(r)^2e^{-ir\log n}\,dr.
\end{equation}
Because \(v\) is even,
\begin{equation}
\tag{7.5}\label{eq:7.5}
\int_{\R}\Phi(r)^2e^{-irx}\,dr
 =2\pi(\phi^2*\phi^2)(x)=2\pi L(v*v)(x/L).
\end{equation}
To make the weighted partial summation explicit, put
\[
 B(y)=\sum_{\log n\le y}\frac{\Lambda(n)^2}{n}
      =\frac12y^2+E(y),\qquad E(y)=O(y),
\]
which is equivalent to
\begin{equation}
\tag{7.6}\label{eq:7.6}
\sum_{n\le y}\frac{\Lambda(n)^2}{n}
 =\frac12(\log y)^2+O(\log y).
\end{equation}
For \(w=v*v\), Stieltjes integration gives
\[
\begin{aligned}
\sum_{n\le X}a_n^2Lw(\log n/L)
 &=\int_{0^-}^{L}Lw(y/L)\,dB(y)\\
 &=L\int_0^L y w(y/L)\,dy
   +[Lw(y/L)E(y)]_{0^-}^{L}\\
 &\hspace{3em}-\int_0^L E(y)w'(y/L)\,dy.
\end{aligned}
\]
Here \(w(1)=0\), \(E(0)=0\), and the last integral is \(O_v(L^2)\).
Consequently
\begin{equation}
\tag{7.7}\label{eq:7.7}
\sum_{n\le X}a_n^2L(v*v)(\log n/L)
 =L^3\int_0^1s(v*v)(s)\,ds+O_v(L^2)
 =\frac{L^3}{2}\mathcal J(v)+O_v(L^2).
\end{equation}
Adding the conjugate ratio term therefore yields
\begin{equation}
\tag{7.8}\label{eq:7.8}
M_{\mathrm{ratio}}
 =\frac{q\Delta L^3}{2\pi}\mathcal J(v)+O_v(qL^2).
\end{equation}
Replacing \(\Delta-|r|\) by \(\Delta\) costs \(O_v(qL^2)\), by \eqref{eq:5.3},
\(\int|r|\Phi(r)^2\,dr=O_v(1)\), and
\(\sum_{n\le X}a_n^2\ll L^2\).

We next treat the two same-sign terms in \eqref{eq:7.3}. If \(n,m\ge2\)
and \(nm\equiv1\pmod q\), then \(nm\ge q+1\). Put \(r=t-t'\). For
fixed \(r\), the common-translate variable
\(u=t'\) ranges over \(J\cap(J-r)\), and the same-sign phase is a unimodular
factor depending on \(r\) times \(e^{-iu\log(nm)}\).  Hence
\begin{equation}
\tag{7.9}\label{eq:7.9}
\left|\int_{J\cap(J-r)}e^{-iu\log(nm)}\,du\right|
 \le\frac2{\log(nm)}\ll Q^{-1}.
\end{equation}
The remaining \(r\)-integral is bounded by
\(\int\Phi(r)^2\,dr=2\pi bL\), so the analytic kernel is
\(O_\lambda(1)\). Inversion modulo \(q\), restricted to \([2,X]\), is a
partial permutation; therefore Cauchy--Schwarz gives
\begin{equation}
\tag{7.10}\label{eq:7.10}
\sum_{\substack{n,m\le X\\nm\equiv1\ (q)}}a_na_m
 \le\sum_{n\le X}a_n^2\ll L^2.
\end{equation}
Both same-sign products together are consequently \(O_\lambda(qL^2)\).

\subsection{Mixed terms and principal subtraction}

Write the exact parity decomposition
\begin{equation}
\tag{7.11}\label{eq:7.11}
\mu_\chi(t)=A_q(t)+\chi(-1)B(t),
\end{equation}
where \(A_q(t)=Q/(2\pi)+O_{A_0}(\log Q)\) and \(B(t)=O(1)\) on \(J\). For sufficiently
large \(q\), \(X<q-1\), and full-family orthogonality gives
\begin{equation}
\tag{7.12}\label{eq:7.12}
\sum_{\chi\bmod q}P_{\chi,L}(t)=0,
\qquad
\sum_{\chi\bmod q}\chi(-1)P_{\chi,L}(t)=0.
\end{equation}
Indeed \(\sum_\chi\chi(-1)\chi(n)=(q-1)\one_{n\equiv-1\ (q)}\), with the
same identity for the conjugate half, and no \(n\le X\) meets that
congruence. Hence all full-family archimedean--prime mixed terms vanish.

The principal character is introduced only as an algebraic bookkeeping term.
Schur localization of \(\Phi(t-t')^2\), together with \eqref{eq:6.4}, gives
\begin{equation}
\tag{7.13}\label{eq:7.13}
\mathcal M[P_0^*,P_0^*]\ll X\Delta L,
\end{equation}
and, with \(\mu_0^*=A_q+B\),
\begin{equation}
\tag{7.14}\label{eq:7.14}
|\mathcal M[\mu_0^*,P_0^*]|+|\mathcal M[P_0^*,\mu_0^*]|
 \ll Q\sqrt X\,\Delta L.
\end{equation}
Since \(\Delta\asymp T\), \(L\asymp Q\), and \(\lambda<1\) is fixed, the
raw second-moment main scale is \(q\Delta LQ^2\asymp qTQ^3\). The ratios of
\eqref{eq:7.13} and \eqref{eq:7.14} to this scale are, respectively,
\[
 O(q^{\lambda-1}Q^{-2}),\qquad
 O(q^{\lambda/2-1}Q^{-1}),
\]
so both terms are \(o(q\Delta LQ^2)\), uniformly in \(T\).

\begin{proposition}[Second moment]\label{prop:second-moment}
Define
\begin{equation}
\tag{7.15}\label{eq:7.15}
c_\lambda(v)=\frac{\lambda a^2}{b+\lambda^2\mathcal J(v)}.
\end{equation}
Then, uniformly for \(T\) satisfying \eqref{eq:2.1},
\begin{equation}
\tag{7.16}\label{eq:7.16}
\sum_{\chi\ne\chi_0}\norm{\widehat G_\chi}_F^2
 =\frac{1-2\theta}{c_\lambda(v)}\mathcal N_q
  +o_{\eta,A_0}(\mathcal N_q).
\end{equation}
\end{proposition}

\begin{proof}
The finite-centre error in Proposition~\ref{prop:finite-second}, the
endpoint-overlap errors, the ratio prime-number-theorem error, and the
reciprocal terms are all \(O(qQ^2)\). The gamma-factor remainder in
\eqref{eq:7.2} is \(O_{\eta,\psi}(q\Delta LQ\log Q)\), while
\eqref{eq:7.13}--\eqref{eq:7.14} are \(o(q\Delta LQ^2)\), uniformly in \(T\).
Combining \eqref{eq:7.2} and \eqref{eq:7.8}, the raw main term is
\begin{equation}
\tag{7.17}\label{eq:7.17}
\frac{q\Delta LQ^2}{2\pi}\{b+\lambda^2\mathcal J(v)\}.
\end{equation}
Since \(G_\chi\) is Hermitian,
\begin{equation}
\tag{7.18}\label{eq:7.18}
\norm{\widehat G_\chi}_F^2
 =\frac{\tr((G_\chi/L)^2)}{a^2L^2}.
\end{equation}
Divide \eqref{eq:7.17} by \(a^2L^2\), use \(L=\lambda Q\),
\(\Delta=(1-2\theta)T\), and
\(\mathcal N_q\sim qTQ/(2\pi)\). This gives the main term in
\eqref{eq:7.16}. Every raw \(O(qQ^2)\) error becomes \(O(q)\) after
\eqref{eq:7.18}, hence has relative size \(O((TQ)^{-1})\ll(\log Q)^{-1-\eta}\).
The normalized gamma remainder is \(O_{\eta,\psi}(qT\log Q)\), of relative
size \(O_{A_0}(\log Q/Q)\). The principal-character terms have the two
relative bounds displayed above. All errors are therefore
\(o_{\eta,A_0}(\mathcal N_q)\), uniformly in \(T\).
\end{proof}

\section{The zero-side matrix and its inertia certificate}

We retain the Fourier convention
\[
\widehat f(t)=\int_{\R}f(u)e^{itu}\,du,
\qquad
f(u)=\frac1{2\pi}\int_{\R}\widehat f(t)e^{-itu}\,dt.
\]
Fix \(\eta>0\) and \(A_0>0\). Let \(0<\lambda<1\),
\(0<\theta<1/4\), and let
\(0\ne\psi\in G_c^{\mathfrak s}((-1/2,1/2))\),
\(\mathfrak s=1+\eta/2\), be real and even. These parameters
are held fixed while \(q\to\infty\), and \(T\) may vary subject to
\eqref{eq:2.1}. Put
\[
Q=\log q,\qquad L=\lambda Q,\qquad I=(T,2T],
\]
\[
J=[(1+\theta)T,(2-\theta)T],\qquad
\phi(u)=\psi(u/L),
\]
\[
a=\int_{-1/2}^{1/2}\psi(s)^2\,ds,
\qquad
\tau_k=\tau_0+\frac{2\pi k}{L}.
\]
Let \(\mathcal K=\{k:\tau_k\in J\}\). For \(z\in\C\), define the
finite sampled vector
\[
u(z)=\bigl(\widehat\phi(z-\tau_k)\bigr)_{k\in\mathcal K}.
\]
If \(\rho=\beta+i\gamma\) is a nontrivial zero of \(L(s,\chi)\), write
\begin{equation}
\tag{8.1}\label{eq:8.1}
\delta_\rho=\beta-\frac12,
\qquad
z_\rho=\frac{\rho-1/2}{i}=\gamma-i\delta_\rho.
\end{equation}

\begin{lemma}[Same-character zero symmetry]\label{lem:same-character-symmetry}
The multiset of nontrivial zeros of \(L(s,\chi)\) is invariant, preserving
multiplicity, under
\[
\rho\longmapsto1-\overline\rho.
\]
In the coordinate \eqref{eq:8.1}, this map is
\(z_\rho\mapsto\overline{z_\rho}\).
\end{lemma}

\begin{proof}
The primitive functional equation in
DLMF~\cite[(25.15.5)]{DLMF}, written in the normalization \eqref{eq:4.1}, is
\(\Lambda(s,\chi)=\varepsilon_\chi\Lambda(1-s,\overline\chi)\).
It sends a zero \(\rho\) of \(\chi\) to the zero \(1-\rho\) of
\(\overline\chi\), with the same multiplicity. The identity
\(\Lambda(s,\overline\chi)=\overline{\Lambda(\overline s,\chi)}\) then
sends that zero to \(1-\overline\rho\) of the original character. Finally,
\[
\frac{1-\overline\rho-1/2}{i}
 =\overline{\frac{\rho-1/2}{i}}.
\]
\end{proof}

The zero-side Weil summand belonging to \(\rho\) is
\[
u(z_\rho)u(\overline{z_\rho})^*.
\]
Consequently a critical-line zero \(1/2+i\gamma\) of multiplicity \(m\)
contributes the positive semidefinite block
\begin{equation}
\tag{8.2}\label{eq:8.2}
m\,u(\gamma)u(\gamma)^*,
\end{equation}
whereas an off-line orbit \(\{\rho,1-\overline\rho\}\) of common
multiplicity \(m\) contributes
\begin{equation}
\tag{8.3}\label{eq:8.3}
m\{u(z)u(\overline z)^*+u(\overline z)u(z)^*\}.
\end{equation}
The latter block is Hermitian. If \(x=u(z)\), \(y=u(\overline z)\), and
\(B=[x\ y]\), then it equals
\[
B\begin{pmatrix}0&1\\1&0\end{pmatrix}B^*.
\]
Sylvester's law of inertia, after restriction to the range of \(B\), gives
at most one positive and at most one negative eigenvalue. No independence of
\(x\) and \(y\) is asserted or needed.

We shall also use the real tight-frame identity
\[
\sum_{k\in\Z}\widehat\phi(t-\tau_k)
 \overline{\widehat\phi(t'-\tau_k)}
 =L\widehat{|\phi|^2}(t-t')\qquad(t,t'\in\R).
\]
At \(t'=t\) it gives
\[
\sum_{k\in\Z}|\widehat\phi(t-\tau_k)|^2=aL^2.
\]
Restriction to \(\mathcal K\) therefore yields
\begin{equation}
\tag{8.4}\label{eq:8.4}
\frac{\norm{u(t)}_2^2}{aL^2}\le1\qquad(t\in\R).
\end{equation}
For the truncated lattice, \eqref{eq:8.4} is the form used below.

\begin{lemma}[Rank--trace inequality]\label{lem:rank-trace}
Let \(P\succeq0\) be Hermitian with \(\rank P\le r\), and let \(R\) be
Hermitian with \(n_+(R)\le b\). Then
\begin{equation}
\tag{8.2.1}\label{eq:8.2.1}
\norm{P+R}_F^2\ge2\tr P-r+4\tr R-4b.
\end{equation}
\end{lemma}

\begin{proof}
It is enough to use the actual rank \(r_0\) and positive index \(b_0\).
Let \(p_1\ge\cdots\ge p_{r_0}>0\) be the nonzero eigenvalues of \(P\),
with corresponding orthonormal eigenvectors \(e_i\), and write the
eigenvalues of \(R\) as \(q_1\le\cdots\le q_d\), with orthonormal
eigenvectors \(f_j\).  For fixed \(k\le r_0\), put
\[
 a_j^{(k)}=\sum_{i=1}^k|\langle e_i,f_j\rangle|^2.
\]
These numbers satisfy \(0\le a_j^{(k)}\le1\) and
\(\sum_ja_j^{(k)}=k\), because they are diagonal entries of an orthogonal
projection in the \(f_j\)-basis.  Since the \(q_j\) are increasing, moving
any available mass from a larger index to a smaller one cannot increase
\(\sum_jq_ja_j^{(k)}\).  Hence its minimum under these constraints is
\(\sum_{j=1}^kq_j\), and therefore
\[
\sum_{i=1}^k\langle Re_i,e_i\rangle
 =\sum_{j=1}^dq_ja_j^{(k)}\ge\sum_{j=1}^kq_j
 \qquad(1\le k\le r_0).
\]
Thus the needed Ky Fan prefix inequality is proved here. Abel summation,
using \(p_k-p_{k+1}\ge0\), now gives the oppositely ordered
trace bound
\[
\tr(PR)\ge\sum_{i=1}^{r_0}p_iq_i.
\]
It follows that the difference between the left side of
\eqref{eq:8.2.1} and its right side with \(r_0,b_0\) is bounded below by
\[
 \sum_{i\le r_0}
 \{p_i^2+2p_iq_i-2p_i+1+q_i^2-4q_i
      +4\one_{q_i>0}\}
 +\sum_{i>r_0}\{q_i^2-4q_i+4\one_{q_i>0}\}.
\]
We compare these summands eigenvalue by eigenvalue. An unpaired
\(q\le0\) contributes \(q^2-4q\ge0\), while an
unpaired \(q>0\) contributes \(q^2-4q+4=(q-2)^2\ge0\). A paired
\((p,q)\) contributes
\begin{equation}
\tag{8.2.2}\label{eq:8.2.2}
p^2+2pq-2p+1+q^2-4q+4\one_{q>0}.
\end{equation}
If \(q=-x\le0\), this is
\[
(p-1)^2+x^2+(4-2p)x.
\]
Its minimum for \(x\ge0\) is \((p-1)^2\) when \(p\le2\), and
\(2p-3\) when \(p\ge2\); both are nonnegative. If \(q>0\), putting
\(s=p+q\) shows that \eqref{eq:8.2.2} is
\[
s^2-2p-4q+5\ge s^2-4s+5=(s-2)^2+1.
\]
Summing proves the result with \(r_0,b_0\). Replacing them by the stated
upper bounds \(r\ge r_0\) and \(b\ge b_0\) only decreases the right-hand
side, proving the asserted form.
\end{proof}

Let \(A_\chi\) denote the sum of the blocks \eqref{eq:8.2}--\eqref{eq:8.3}
whose common ordinate \(\gamma\) belongs to \(I=(T,2T]\), and put
\(\widehat A_\chi=A_\chi/(aL^2)\).  For a real ordinate define
\[
 r_\gamma=\frac{\|u(\gamma)\|_2^2}{aL^2},
 \qquad 0\le r_\gamma\le1
\]
by \eqref{eq:8.4}.  A critical-line zero of multiplicity \(m\) contributes
a positive semidefinite block with positive index at most one and normalized
trace \(mr_\gamma\).  For an off-line orbit, with
\(x=u(z)\) and \(y=u(\overline z)\), the normalized trace is
\[
 \frac{2m\operatorname{Re}\langle y,x\rangle}{aL^2},
\]
which has no fixed sign, while the positive index is still at most one.

Partition the distinct zeros in \(I\) into \(s_1\) simple zeros on the
critical line, \(s_2\) distinct multiple zeros on the critical line of multiplicities
\(m_j\ge2\), and \(p\) distinct off-line functional-equation orbits of
common multiplicities \(n_\ell\ge1\).  The involution
\(\rho\mapsto1-\overline\rho\) preserves the ordinate and multiplicity.
If an off-line point satisfied \(\rho=1-\overline\rho\), then
\(\Re\rho=1/2\), a contradiction.  Hence every off-line orbit contains
exactly two distinct zero points, both belonging to \(I\).  Therefore
\begin{equation}
\tag{8.5}\label{eq:8.5}
\begin{aligned}
\mathcal N_\chi(I)&=s_1+\sum_jm_j+2\sum_\ell n_\ell,\\
\mathcal N^s_{0,\chi}(I)&=s_1,\\
\mathcal N^*_{0,\chi}(I)&=s_1+s_2,\\
\mathcal N_{d,\chi}(I)&=s_1+s_2+2p.
\end{aligned}
\end{equation}

\begin{proposition}[Certificates for simple and distinct zeros]\label{prop:zero-certificate}
For every nonprincipal \(\chi\pmod q\), writing
\[
 M_\chi=4\tr\widehat A_\chi-\|\widehat A_\chi\|_F^2,
\]
one has
\begin{align}
\mathcal N^s_{0,\chi}(I)
 &\ge M_\chi-2\mathcal N_\chi(I),
 \tag{8.3.1}\label{eq:8.3.1}\\
\mathcal N^*_{0,\chi}(I)
 &\ge M_\chi-2\mathcal N_\chi(I),
 \tag{8.3.1a}\label{eq:8.3.1a}\\
\mathcal N_{d,\chi}(I)
 &\ge\frac12\{M_\chi-\mathcal N_\chi(I)\}.
 \tag{8.3.1b}\label{eq:8.3.1b}
\end{align}
\end{proposition}

\begin{proof}
Put the simple critical-line blocks into \(P\) and all remaining blocks
into \(R\). By \eqref{eq:8.4},
\[
P\succeq0,
\qquad
\rank P\le s_1,
\qquad
\tr P\le s_1.
\]
The positive index is subadditive under Hermitian addition: if
\(n_+(R_j)\le b_j\), intersecting nonpositive subspaces of codimensions
\(b_j\) gives \(n_+(\sum_jR_j)\le\sum_jb_j\).  Each multiple
critical-line block and each off-line orbit block has positive index at most
one, so \(n_+(R)\le s_2+p\). Lemma~\ref{lem:rank-trace} therefore implies
\begin{equation}
\tag{8.3.2}\label{eq:8.3.2}
M_\chi\le3s_1+4s_2+4p.
\end{equation}
The four counting categories in \eqref{eq:8.5} now give
\[
\begin{aligned}
s_1+2\mathcal N_\chi(I)
 &=3s_1+2\sum_jm_j+4\sum_\ell n_\ell
 \ge3s_1+4s_2+4p,\\
\mathcal N^*_{0,\chi}(I)+2\mathcal N_\chi(I)
 &=3s_1+s_2+2\sum_jm_j+4\sum_\ell n_\ell
 \ge3s_1+4s_2+4p,\\
2\mathcal N_{d,\chi}(I)+\mathcal N_\chi(I)
 &=3s_1+2s_2+4p+\sum_jm_j+2\sum_\ell n_\ell
 \ge3s_1+4s_2+4p.
\end{aligned}
\]
Combining these inequalities with \eqref{eq:8.3.2} proves
\eqref{eq:8.3.1}--\eqref{eq:8.3.1b}.
\end{proof}

\section{Gevrey complex sampling and the exterior zero tail}

We retain the good--bad terminology and the exceptional set \(\mathcal B_q\)
from Lemma~\ref{lem:exceptional}. The complex polarization of the infinite lattice identity is
\begin{equation}
\tag{9.1}\label{eq:9.1}
\sum_{k\in\Z}\widehat\phi(z-\tau_k)
 \overline{\widehat\phi(z'-\tau_k)}
 =L\widehat{|\phi|^2}(z-\overline{z'}).
\end{equation}
For complex sampling we use the Gevrey derivative bounds in
\eqref{eq:gevrey-definition} to obtain quantitative Fourier--Laplace decay.

\begin{lemma}[Gevrey complex sampled-vector bound]
\label{lem:complex-vector}
Let \(\mathfrak s=1+\eta/2\). There exist constants \(c,C>0\), depending
only on the fixed window and the lattice spacing, such that, uniformly for
\(r,\delta\in\R\),
\begin{equation}
\tag{9.1.1}\label{eq:9.1.1}
\norm{u(r-i\delta)}_2^2
 \le C L^2e^{L|\delta|}
 \exp\!\left[-c\{L\dist(r,J)\}^{1/\mathfrak s}\right].
\end{equation}
Consequently,
\begin{equation}
\tag{9.1.2}\label{eq:9.1.2}
\norm{u(r-i\delta)}_2\norm{u(r+i\delta)}_2
 \le C L^2e^{L|\delta|}
 \exp\!\left[-c\{L\dist(r,J)\}^{1/\mathfrak s}\right].
\end{equation}
\end{lemma}

\begin{proof}
By \eqref{eq:gevrey-definition}, after increasing the constants if
necessary,
\[
 \|\psi^{(m)}\|_1\le C_0R_0^m(m!)^{\mathfrak s}
 \qquad(m\ge0).
\]
For \(w=r-i\delta-\tau\), the substitution \(u=Ls\) and \(m\) integrations
by parts against the whole complex exponential give
\[
 |\widehat\phi(w)|
 \le L e^{L|\delta|/2}
 \min\left\{\|\psi\|_1,
 \frac{C_0R_0^m(m!)^{\mathfrak s}}{(L|w|)^m}\right\}.
\]
Since \(|w|\ge |r-\tau|\), the denominator may be weakened to
\(L|r-\tau|\). Using \(m!\le m^m\) and taking
\(m=\lfloor c_0(L|r-\tau|)^{1/\mathfrak s}\rfloor\), with \(c_0>0\)
sufficiently small in terms of \(R_0\), yields
\begin{equation}
\tag{9.1.3}\label{eq:9.1.3}
 |\widehat\phi(r-i\delta-\tau)|
 \le C_1Le^{L|\delta|/2}
 \exp\!\left[-c_1\{L|r-\tau|\}^{1/\mathfrak s}\right].
\end{equation}
The trivial estimate covers the bounded range omitted by the integer choice
of \(m\).

The scaled lattice points \(L\tau_k\) have spacing \(2\pi\). If
\(D=L\dist(r,J)\), then comparison with an integral gives
\[
 \sum_{k\in\mathcal K}
 \exp\!\left[-2c_1\{L|r-\tau_k|\}^{1/\mathfrak s}\right]
 \le C_2e^{-c_2D^{1/\mathfrak s}}.
\]
Indeed the left side is bounded for \(D\le1\); for \(D>1\), the integral
tail is a polynomial in \(D\) times \(e^{-2c_1D^{1/\mathfrak s}}\), and
that polynomial is absorbed by weakening the exponential constant. Squaring
\eqref{eq:9.1.3} and summing proves \eqref{eq:9.1.1}. Applying the same
estimate to both signs of \(\delta\) and taking square roots proves
\eqref{eq:9.1.2}.
\end{proof}

Lemma~\ref{lem:local-zero-count}, deduced from the symmetric zero-count
theorem of Bennett--Martin--O'Bryant--Rechnitzer, gives uniformly for
primitive characters and real \(t\),
\begin{equation}
\tag{9.2}\label{eq:9.2}
\#\{\rho:|\Im\rho-t|\le1\}\ll\log(q(|t|+3)),
\end{equation}
with multiplicity. The trace norm of a rank-one operator is
\(\norm{xy^*}_1=\norm x_2\norm y_2\). An off-line orbit has two rank-one
summands; the resulting factor two is absorbed in the implied constant.
Lemma~\ref{lem:complex-vector} therefore bounds an on-line block or a full
off-line orbit block of multiplicity \(m\) by
\begin{equation}
\tag{9.3}\label{eq:9.3}
 \ll mL^2e^{L|\delta|}
 \exp\!\left[-c\{L\dist(\gamma,J)\}^{1/\mathfrak s}\right].
\end{equation}
For fixed \(q\), \(|\delta|<1/2\); the exponential decay in
\(|\gamma|^{1/\mathfrak s}\), together with \eqref{eq:9.2}, shows that
the sum of these trace norms over all orbits converges. Hence the zero-side
matrix is absolutely convergent in trace norm and may be grouped
into functional-equation orbits and split according to \(\gamma\in I\).

Define \(E_\chi\) to be the trace-norm sum over the complementary set
\(\gamma\notin I\). Thus, with the endpoint convention fixed above,
\begin{equation}
\tag{9.4}\label{eq:9.4}
E_\chi=G_\chi-A_\chi.
\end{equation}

\begin{proposition}[Good-character Gevrey zero tail]
\label{prop:good-tail}
Fix \(B>0\) after \(\eta,A_0,\lambda,\theta,\psi\). Uniformly for good
characters and for \(T\) satisfying \eqref{eq:2.1},
\begin{equation}
\tag{9.2.1}\label{eq:9.2.1}
\norm{\frac{E_\chi}{aL^2}}_1
 \ll_{\eta,A_0,B,\lambda,\theta,\psi}Q^{-B}.
\end{equation}
\end{proposition}

\begin{proof}
Put
\[
 \alpha=\frac1{\mathfrak s},\qquad
 \beta=\frac{1+\eta}{\mathfrak s}>1,
 \qquad \mathcal H=TQ.
\]
For every exterior ordinate \(\gamma\notin I=(T,2T]\),
\begin{equation}
\tag{9.2.2}\label{eq:9.2.2}
 \dist(\gamma,J)\ge\theta T,
 \qquad L\dist(\gamma,J)\ge\lambda\theta\mathcal H.
\end{equation}
By summing \eqref{eq:9.2} over unit intervals,
\begin{equation}
\tag{9.2.3}\label{eq:9.2.3}
 \sum_{|\gamma|\le4T+1}m_\rho
 \ll (T+1)\{Q+\log(T+3)\}
 \ll_{A_0}(T+1)Q.
\end{equation}
We first treat these local ordinates. If \(|\delta|<\delta_0\), then
\(e^{L|\delta|}\le Q^{24\lambda}\); hence their normalized contribution
is
\begin{equation}
\tag{9.2.4}\label{eq:9.2.4}
 \ll_{A_0}(T+1)Q^{24\lambda+1}
 \exp[-c\mathcal H^\alpha].
\end{equation}
If \(|\delta|\ge\delta_0\), goodness gives
\[
 |\gamma|>q^{|\delta|/24}\ge Q,
 \qquad e^{L|\delta|}<|\gamma|^{24\lambda}.
\]
Thus the local far-line contribution is
\begin{equation}
\tag{9.2.5}\label{eq:9.2.5}
 \ll_{A_0}(T+1)(4T+1)^{24\lambda}Q
 \exp[-c\mathcal H^\alpha].
\end{equation}
Because \(T\le Q^{A_0}\) and
\[
 \mathcal H^\alpha
 \ge(\log Q)^{(1+\eta)/\mathfrak s}
 = (\log Q)^\beta,
 \qquad \beta>1,
\]
both \eqref{eq:9.2.4} and \eqref{eq:9.2.5} are \(O(Q^{-B})\) for every
fixed \(B>0\).

It remains to treat \(|\gamma|>4T+1\). In that range
\begin{equation}
\tag{9.2.6}\label{eq:9.2.6}
 \dist(\gamma,J)\ge c_\theta|\gamma|.
\end{equation}
For near-line zeros, unit-shell summation gives
\begin{equation}
\tag{9.2.7}\label{eq:9.2.7}
 \ll Q^{24\lambda}
 \sum_{n\ge1}(Q+\log(n+3))
 \exp[-c(Ln)^\alpha]
 =O_B(Q^{-B}).
\end{equation}
For far-line zeros, goodness forces \(|\gamma|>Q\), and the same argument
gives
\begin{equation}
\tag{9.2.8}\label{eq:9.2.8}
 \ll\sum_{n\ge Q-1}(Q+\log(n+3))(n+1)^{24\lambda}
 \exp[-c(Ln)^\alpha]
 =O_B(Q^{-B}).
\end{equation}
The last equalities in \eqref{eq:9.2.7}--\eqref{eq:9.2.8} follow by
comparison with the corresponding integrals: since \(L\asymp Q\), the
first exponential already begins with \(e^{-cQ^\alpha}\), and the second
with \(e^{-cQ^{2\alpha}}\); either dominates every fixed polynomial in
\(Q\) and in the shell index. Negative ordinates are included in the same
estimates. A real exceptional zero either makes \(\chi\) bad or belongs to
the near-line local contribution. Summing the four pieces proves
\eqref{eq:9.2.1}.
\end{proof}

\section{Proof of the main reduction: assembly}

Let
\[
\mathcal N_q=\sum_{\chi\ne\chi_0}\mathcal N_\chi(I).
\]
Corollary~\ref{cor:counting-deletion-interface} gives, uniformly in \(T\),
\begin{equation}
\tag{10.1}\label{eq:10.1}
\mathcal N_q=\frac{qTQ}{2\pi}+O(qT\log(T+2))+O(q)
 =\frac{qTQ}{2\pi}\{1+o_{\eta,A_0}(1)\}.
\end{equation}
The endpoint-safe derivation is given in
Lemma~\ref{lem:half-open-extraction} and \eqref{eq:3.4}: a common right shift
of \(T\) and \(2T\) extracts \(T<|\gamma|\le2T\) exactly, and family
conjugation then identifies the positive count with one half of that symmetric
count.

The prime-side moment calculation and the finite-centre replacement proved
earlier give, with
\[
c_\lambda(v)=\frac{\lambda a(v)^2}{b(v)+\lambda^2\mathcal J(v)},
\qquad v=\psi^2,
\]
the two formulas
\begin{equation}
\tag{10.2}\label{eq:10.2}
\sum_{\chi\ne\chi_0}\tr\widehat G_\chi
 =(1-2\theta)\mathcal N_q+o_{\eta,A_0}(\mathcal N_q),
\end{equation}
\begin{equation}
\tag{10.3}\label{eq:10.3}
\sum_{\chi\ne\chi_0}\|\widehat G_\chi\|_F^2
 =\frac{1-2\theta}{c_\lambda(v)}\mathcal N_q
  +o_{\eta,A_0}(\mathcal N_q).
\end{equation}
Put
\[
M^G_\chi=4\tr\widehat G_\chi-\|\widehat G_\chi\|_F^2,
\quad
C^{(0)}_\chi=M^G_\chi-2\mathcal N_\chi(I),
\quad
C^{(d)}_\chi=\frac12\{M^G_\chi-\mathcal N_\chi(I)\}.
\]
The same quantity \(C^{(0)}_\chi\) serves both the simple and the distinct
critical-line certificates. If
\(d=\#\mathcal K=\Delta L/(2\pi)+O(1)=O(TQ)\), then every eigenvalue
\(x\) of \(\widehat G_\chi\) satisfies \(4x-x^2\le4\), whence
\begin{equation}
\tag{10.4}\label{eq:10.4}
C^{(0)}_\chi\le4d,
\qquad
C^{(d)}_\chi\le2d.
\end{equation}
Lemma~\ref{lem:exceptional} and \eqref{eq:10.1} give
\[
 \frac{d|\mathcal B_q|}{\mathcal N_q}
 \ll\frac{(TQ)(q/Q^2)}{qTQ}\ll Q^{-2}.
\]
Therefore
\begin{equation}
\tag{10.5}\label{eq:10.5}
\begin{aligned}
\sum_{\chi\ \mathrm{good}}C^{(0)}_\chi
 &\ge\sum_{\chi\ne\chi_0}C^{(0)}_\chi-4d|\mathcal B_q|
  =\sum_{\chi\ne\chi_0}C^{(0)}_\chi+o_{\eta,A_0}(\mathcal N_q),\\
\sum_{\chi\ \mathrm{good}}C^{(d)}_\chi
 &\ge\sum_{\chi\ne\chi_0}C^{(d)}_\chi-2d|\mathcal B_q|
  =\sum_{\chi\ne\chi_0}C^{(d)}_\chi+o_{\eta,A_0}(\mathcal N_q).
\end{aligned}
\end{equation}
Only the upper caps \eqref{eq:10.4}, not absolute bounds, are needed.

For good \(\chi\), write
\(\widehat A_\chi=\widehat G_\chi-\widehat E_\chi\).
Proposition~\ref{prop:good-tail} gives
\[
|\tr(\widehat A_\chi-\widehat G_\chi)|
 \le\|\widehat E_\chi\|_1,
\]
and
\begin{equation}
\tag{10.6}\label{eq:10.6}
\left|\|\widehat A_\chi\|_F^2
      -\|\widehat G_\chi\|_F^2\right|
 \le2\|\widehat G_\chi\|_F\|\widehat E_\chi\|_1
  +\|\widehat E_\chi\|_1^2.
\end{equation}
By \eqref{eq:10.3} and family Cauchy--Schwarz,
\[
\sum_\chi\|\widehat G_\chi\|_F
 =O\!\left(q\sqrt{TQ}\right).
\]
Choosing any fixed \(B>1\) in Proposition~\ref{prop:good-tail}, the
trace transfer costs \(O(qQ^{-B})\). The quadratic cross term costs
\(O(qQ^{-B}\sqrt{TQ})\), and the squared tail costs \(O(qQ^{-2B})\).
Relative to \(\mathcal N_q\asymp qTQ\), these are, respectively,
\[
 O\!\left(\frac{Q^{-B}}{TQ}\right),\qquad
 O\!\left(\frac{Q^{-B}}{\sqrt{TQ}}\right),\qquad
 O\!\left(\frac{Q^{-2B}}{TQ}\right),
\]
which are \(o_{\eta,A_0}(1)\) uniformly in \(T\). Thus summing
\eqref{eq:10.6} replaces \(M^G_\chi\) by
\(M^A_\chi=4\tr\widehat A_\chi-\|\widehat A_\chi\|_F^2\) over the good
characters at total cost \(o_{\eta,A_0}(\mathcal N_q)\).

Sum the three inequalities of Proposition~\ref{prop:zero-certificate} over
the good characters and use the preceding estimates.  In the negative
zero-count terms, replace the good-subfamily multiplicity by the larger
quantity \(\mathcal N_q\).  We obtain
\begin{equation}
\tag{10.7}\label{eq:assembly}
\begin{aligned}
\sum_{\chi\ne\chi_0}\mathcal N^s_{0,\chi}(I)
 &\ge\left\{(1-2\theta)\left(4-\frac1{c_\lambda(v)}\right)
 -2+o_{\eta,A_0}(1)\right\}\mathcal N_q,\\
\sum_{\chi\ne\chi_0}\mathcal N^*_{0,\chi}(I)
 &\ge\left\{(1-2\theta)\left(4-\frac1{c_\lambda(v)}\right)
 -2+o_{\eta,A_0}(1)\right\}\mathcal N_q,\\
\sum_{\chi\ne\chi_0}\mathcal N_{d,\chi}(I)
 &\ge\frac12\left\{(1-2\theta)
       \left(4-\frac1{c_\lambda(v)}\right)-1+o_{\eta,A_0}(1)\right\}\mathcal N_q.
\end{aligned}
\end{equation}
Each left side was enlarged from the good subfamily to the whole family,
which is legitimate because these are nonnegative counts.

\begin{proof}[Proof of Proposition~\ref{prop:fixed-auxiliary}]
Divide the three inequalities in \eqref{eq:assembly} by the positive quantity
\(\mathcal N_q\). Equation \eqref{eq:10.1} and every preceding error bound
are uniform for \(T\) satisfying \eqref{eq:2.1}, which gives the three
assertions of Proposition~\ref{prop:fixed-auxiliary}.
\end{proof}

\section{Proof of the main theorem: variational optimization}

We finally optimize the quadratic quotient appearing in
\eqref{eq:assembly} over nonzero nonnegative
\(v\in L^2([-1/2,1/2])\). Scaling \(v\) does not change
\(c_\lambda(v)\), so we may impose \(\int v=1\). Let
\[
(Kv)(s)=\int_{-1/2}^{1/2}|s-t|v(t)\,dt.
\]
The Euler equation for minimizing
\[
 b(v)+\lambda^2\mathcal J(v)
 =\langle v,(\mathrm{Id}+\lambda^2K)v\rangle
\]
under the integral constraint is
\begin{equation}
\tag{11.1}\label{eq:11.1}
v+\lambda^2Kv=C.
\end{equation}
Since \((Kv)''=2v\) in the interior, every even solution satisfies
\[
v''+2\lambda^2v=0.
\]
The even extremizer is therefore, up to a positive scalar,
\begin{equation}
\tag{11.2}\label{eq:11.2}
v_\lambda^*(s)=\cos(\sqrt2\lambda s).
\end{equation}
Substitution into \eqref{eq:11.1}, or direct evaluation of the three
functionals, gives
\begin{equation}
\tag{11.3}\label{eq:11.3}
c_\lambda^*=\sup_v c_\lambda(v)
 =\frac{\sqrt2\tan(\lambda/\sqrt2)}
 {1+(\lambda/\sqrt2)\tan(\lambda/\sqrt2)}.
\end{equation}
We justify that the stationary point is the global minimizer. If \(f\) is a
real zero-mean variation and
\(F(s)=\int_{-1/2}^sf(t)\,dt\), then
\(F(-1/2)=F(1/2)=0\), and two integrations by parts give the
conditional-negative-definiteness identity
\[
\iint|s-t|f(s)f(t)\,ds\,dt
 =-2\int_{-1/2}^{1/2}F(s)^2\,ds.
\]
The required Dirichlet Wirtinger inequality follows directly:
expanding \(F\in H_0^1(-1/2,1/2)\) as
\(F(s)=\sum_{n\ge1}b_n\sin(n\pi(s+1/2))\), Parseval gives
\[
 \int F^2=\frac12\sum_{n\ge1}|b_n|^2,
 \qquad
 \int(F')^2=\frac{\pi^2}{2}\sum_{n\ge1}n^2|b_n|^2,
\]
so \(\int F^2\le\pi^{-2}\int(F')^2\).  Consequently
\[
\begin{aligned}
\int f^2+\lambda^2\iint|s-t|f(s)f(t)\,ds\,dt
 &=\int(F')^2-2\lambda^2\int F^2\\
 &\ge\left(1-\frac{2\lambda^2}{\pi^2}\right)\int f^2>0
\end{aligned}
\]
for \(0<\lambda\le1\). Hence the objective is strictly convex on the
affine hyperplane \(\int v=1\), and the Euler solution is its unique global
minimizer. Moreover \(v_\lambda^*>0\) on \([-1/2,1/2]\) for
\(0<\lambda\le1\).

The function in \eqref{eq:11.2} is not compactly supported in the open
interval. We approximate it inside the fixed Gevrey class used in
Section~2. Put
\[
 g_{\mathfrak s}(x)=
 \begin{cases}
 0,&x\le0,\\
 \exp\!\left(-x^{-1/(\mathfrak s-1)}\right),&x>0,
 \end{cases}
 \qquad
 \Theta_{\mathfrak s}(x)=
 \frac{g_{\mathfrak s}(x)}
 {g_{\mathfrak s}(x)+g_{\mathfrak s}(1-x)}.
\]
A direct Fa\`a di Bruno estimate gives constants \(C,R>0\), depending
only on \(\mathfrak s\), such that
\[
 \|\Theta_{\mathfrak s}^{(m)}\|_\infty
 \le CR^m(m!)^{\mathfrak s}\qquad(m\ge0).
\]
Indeed the same estimate first applies to the flat function
\(g_{\mathfrak s}\); on \([0,1]\) the denominator in the definition of
\(\Theta_{\mathfrak s}\) is bounded away from zero, and the Gevrey class
is closed under multiplication and reciprocals of nonvanishing functions.
Moreover \(\Theta_{\mathfrak s}=0\) on \(( -\infty,0]\) and
\(\Theta_{\mathfrak s}=1\) on \([1,\infty)\). Define the even cutoffs
\[
 \zeta_n(s)=
 \Theta_{\mathfrak s}(n(s+1/2)-1)
 \Theta_{\mathfrak s}(n(1/2-s)-1).
\]
Then \(\zeta_n\in G_c^{\mathfrak s}((-1/2,1/2))\),
\(0\le\zeta_n\le1\), it is supported in
\([-1/2+1/n,1/2-1/n]\), and \(\zeta_n=1\) on
\([-1/2+2/n,1/2-2/n]\) for \(n\ge5\). Set
\[
\psi_n=(v_\lambda^*)^{1/2}\zeta_n,
\qquad
v_n=\psi_n^2=v_\lambda^*\zeta_n^2.
\]
Since \((v_\lambda^*)^{1/2}\) is analytic and positive on a neighbourhood
of the closed interval, each \(\psi_n\) belongs to
\(G_c^{\mathfrak s}((-1/2,1/2))\). Moreover
\(v_n\to v_\lambda^*\) in \(L^1\cap L^2\), so
\(a(v_n)\to a(v_\lambda^*)\) and \(b(v_n)\to b(v_\lambda^*)\). Also,
since \(|s-t|\le1\),
\[
|\mathcal J(f)-\mathcal J(g)|
 \le\norm{f-g}_1(\norm f_1+\norm g_1),
\]
and hence \(\mathcal J(v_n)\to\mathcal J(v_\lambda^*)\). Thus admissible
windows approach the supremum \eqref{eq:11.3}.

\begin{proposition}[Fixed-bandwidth conclusion]\label{prop:fixed-bandwidth}
For every fixed \(\eta>0\), \(A_0>0\), and every fixed
\(0<\lambda<1\), uniformly
for \(T\) satisfying \eqref{eq:2.1},
\begin{align*}
\frac{\mathcal N^s_{0,q}}{\mathcal N_q},\quad
\frac{\mathcal N^*_{0,q}}{\mathcal N_q}
&\ge2-\frac{1+(\lambda/\sqrt2)\tan(\lambda/\sqrt2)}
 {\sqrt2\tan(\lambda/\sqrt2)}-o_{\eta,A_0}(1),\\
\frac{\mathcal N_{d,q}}{\mathcal N_q}
&\ge\frac32-\frac{1+(\lambda/\sqrt2)\tan(\lambda/\sqrt2)}
 {2\sqrt2\tan(\lambda/\sqrt2)}-o_{\eta,A_0}(1).
\end{align*}
\end{proposition}

\begin{proof}
Fix \(\eta,A_0\) and \(\lambda\), and let \(\varepsilon>0\). First choose
one \(\theta>0\), and then one Gevrey window from the fixed approximating
sequence above, so that the resulting constant is within \(\varepsilon\)
of the supremum in \eqref{eq:11.3}. These choices are independent of
\(q\) and \(T\). Proposition~\ref{prop:fixed-auxiliary} then applies with
an error uniform over the full height range. Letting \(q\to\infty\) and
finally \(\varepsilon\downarrow0\) gives all three displayed bounds. At no
stage does an auxiliary parameter depend on \(q\) or on \(T\).
\end{proof}

The three left sides of Proposition~\ref{prop:fixed-bandwidth} are
independent of \(\lambda\). To pass to the endpoint without introducing a
\(q\)-dependent bandwidth, fix \(\varepsilon>0\), choose one
\(\lambda<1\) for which the corresponding constant is within
\(\varepsilon\) of its supremum, and then let \(q\to\infty\) in the uniform
fixed-\(\lambda\) estimate. Since \(\varepsilon\) is arbitrary, this is
equivalent to letting \(\lambda\uparrow1\) only after the fixed-bandwidth
theorem has been proved. For the primary simple-zero statistic this gives,
uniformly in \(T\),
\begin{equation}
\tag{11.4}\label{eq:11.4}
\begin{aligned}
\frac{\mathcal N^s_{0,q}}{\mathcal N_q}
&\ge2-\frac1{c_1^*}-o_{\eta,A_0}(1)\\
&=\frac32-\frac1{\sqrt2}\cot\!\left(\frac1{\sqrt2}\right)-o_{\eta,A_0}(1)\\
&=C_{\mathrm{MT}}-o_{\eta,A_0}(1).
\end{aligned}
\end{equation}
The two distinct-zero bounds, again uniformly in \(T\), are
\begin{equation}
\tag{11.5}\label{eq:11.5}
\begin{aligned}
\frac{\mathcal N^*_{0,q}}{\mathcal N_q}
&\ge C_{\mathrm{MT}}-o_{\eta,A_0}(1),\\
\frac{\mathcal N_{d,q}}{\mathcal N_q}
&\ge\frac{3-1/c_1^*}{2}-o_{\eta,A_0}(1)
 =\frac{1+C_{\mathrm{MT}}}{2}-o_{\eta,A_0}(1)\\
&=\frac54-\frac1{2\sqrt2}\cot\!\left(\frac1{\sqrt2}\right)-o_{\eta,A_0}(1)
 =0.836250351839\ldots-o_{\eta,A_0}(1).
\end{aligned}
\end{equation}
The constant in \eqref{eq:11.4} is the classical Montgomery--Taylor constant
\cite{Montgomery1975,ConreyGhoshGonek1998}. In the present setting it arises
after the uniform good--bad decomposition, the finite Gabor moment
calculation, and the inertia certificate have all been completed.
Equation \eqref{eq:11.5} uses no new analytic estimate: it is the exact
finite-dimensional consequence of counting a multiple critical-line block
once and an off-line orbit twice. Equations \eqref{eq:11.4}--\eqref{eq:11.5},
together with \eqref{eq:3.4}, prove Theorem~\ref{thm:main}.

\section{Uniformity and endpoint limitations}

The table records the error scales used in the final assembly. Throughout,
\(\eta>0\), \(A_0>0\), \(0<\lambda<1\), \(\theta\), and the Gevrey
window are fixed before \(q\to\infty\), while \(T\) varies over
\eqref{eq:2.1}. The denominator satisfies
\(\mathcal N_q\asymp qTQ=q\mathcal H\), uniformly. The tail exponent
\(B>1\) is fixed only after the window. All constants are independent of
\(q\) and \(T\).

\begin{center}
\small
\renewcommand{\arraystretch}{1.15}
\begin{tabularx}{\textwidth}{@{}>{\raggedright\arraybackslash}p{0.32\textwidth}Y@{}}
\toprule
Source & Uniform relative bound, unless marked per character\\
\midrule
Denominator: gamma and argument remainders
 & \(O(\log(T+2)/Q)+O(\mathcal H^{-1})\)\\
First moment: archimedean remainder
 & \(O_{A_0}(\log Q/Q)+O(q^{-1})\)\\
First moment: finite centres
 & \(O(\mathcal H^{-1})\)\\
First moment: principal subtraction
 & \(O(q^{\lambda/2-1}Q^{-1})\)\\
Second moment: finite centres and endpoint overlap
 & \(O(\mathcal H^{-1})\)\\
Second moment: gamma-factor remainder
 & \(O_{A_0}(\log Q/Q)\)\\
Second moment: ratio PNT and reciprocal congruences
 & \(O(\mathcal H^{-1})\)\\
Second moment: principal subtraction
 & \(O(q^{\lambda-1}Q^{-2})+O(q^{\lambda/2-1}Q^{-1})\)\\
Exceptional-character deletion
 & \(O(Q^{-2})\)\\
Local Gevrey exterior tail
 & per character:
 \(Q^{C_{\lambda,A_0}}\exp[-c\mathcal H^{1/\mathfrak s}]\)\\
Remote Gevrey exterior tail
 & per character: \(O_B(Q^{-B})\)\\
Trace-norm transfer: trace, cross, and square terms
 & \(O(Q^{-B}/\mathcal H)\),
   \(O(Q^{-B}/\sqrt{\mathcal H})\), and
   \(O(Q^{-2B}/\mathcal H)\)\\
\bottomrule
\end{tabularx}
\end{center}

Because
\[
 \mathcal H\ge(\log Q)^{1+\eta},\qquad
 \mathfrak s=1+\frac\eta2,
\]
one has
\[
 \mathcal H^{1/\mathfrak s}
 \ge(\log Q)^{(1+\eta)/(1+\eta/2)},
\]
whose exponent is strictly larger than one. Hence the local Gevrey tail is
smaller than every fixed power of \(Q^{-1}\), even after multiplication by
the polynomial factor allowed by \(T\le Q^{A_0}\). Also
\(\mathcal H^{-1}\le(\log Q)^{-1-\eta}\) and
\(\log(T+2)\ll_{A_0}\log Q\). Every entry in the table therefore tends to
zero uniformly over the height range. No displayed little-oh estimate is
uniform as \(\eta\downarrow0\), \(A_0\to\infty\),
\(\lambda\uparrow1\), or \(\theta\downarrow0\); these limits are taken
only in the order specified in Section~2.

\subsection*{The lower endpoint \texorpdfstring{\(TQ\asymp1\)}{TQ comparable to one}}
The Gevrey estimate removes the fixed-power requirement
\(TQ\ge Q^\eta\) that would result from a direct polynomial-decay
argument, but it does not reach a window containing only
\(O(1)\) local mean spacings. At \(TQ\asymp1\), the main family scale is
only \(q\), so the \(O(q)\) averaged-argument and finite-centre errors are
no longer little-oh terms. Also \(d\asymp TQ\asymp1\), and the
Fourier--Laplace suppression in Proposition~\ref{prop:good-tail} is only a
fixed constant. Thus the denominator asymptotic, finite-dimensional
approximation to the variational profile, and trace-norm localization all
lose their vanishing margins. Reaching that scale would require a genuinely
finite-normalized-window certificate and a sharper family counting input;
choosing any fixed Gevrey order $\mathfrak s>1$ or increasing the number of
integrations by parts does not repair it.

\subsection*{Uniformity through \texorpdfstring{\(T\asymp Q\)}{T comparable to Q}}
The earlier shortcut available under \(T=o(Q)\) replaces the distance from a
far-line zero to \(J\) by a fixed multiple of its ordinate. The proof of
Proposition~\ref{prop:good-tail} instead splits the exterior zeros into the
local region \(|\gamma|\le4T+1\) and the remote region
\(|\gamma|>4T+1\). In the local region the fixed margin gives
\(\dist(\gamma,J)\ge\theta T\), and the Gevrey factor suppresses the
polynomial cost in \(T\). In the remote region one recovers
\(\dist(\gamma,J)\gg|\gamma|\). This decomposition is uniform for every fixed upper exponent \(A_0\), so
the same tail estimate remains valid when \(T\asymp\log q\) and throughout
the stated polylogarithmic range.

\section{Conclusion}
Fix \(\eta>0\) and \(A_0>0\). Uniformly throughout
\[
 \frac{(\log\log q)^{1+\eta}}{\log q}
 \le T\le(\log q)^{A_0},
 \qquad I=(T,2T],
\]
the family zero count satisfies
\[
 \mathcal N_q=\frac{qT\log q}{2\pi}\{1+o_{\eta,A_0}(1)\}.
\]
Relative to this denominator, the lower proportions of simple zeros and
distinct zeros on the critical line are at least
\(C_{\mathrm{MT}}=0.672500703679\ldots\), while the lower proportion of all
distinct zeros in \(I\), without restriction to the critical line, is at
least \(C_d=(1+C_{\mathrm{MT}})/2=0.836250351839\ldots\), in each case up
to a uniform \(o_{\eta,A_0}(1)\) term. The theorem covers every fixed-power
scale \(T=(\log q)^\xi\) with \(\xi>-1\), as well as scales approaching
\((\log q)^{-1}\) within a factor \((\log\log q)^{1+\eta}\). The proof
combines the endpoint-safe family count and exceptional-character deletion
with finite Gabor moments, Gevrey trace-norm localization, and the zero-side
inertia/rank--trace certificate. The result remains a family-level theorem
for one prime modulus; no individual-character conclusion is asserted.

\section*{Acknowledgements}
The authors thank Fredrik Pr\"uzelius for drawing their attention to the
Alp\"oge--Furman preprint and for communicating his related manuscript. His
questions about the height range prompted the authors to re-examine the
uniformity of their argument.

\section*{Declarations}

\textit{Funding.} The authors received no specific funding for this work.

\textit{Data availability.} Data sharing is not applicable to this article.

\end{document}